\documentclass[11pt]{article}

\usepackage{enumerate, fullpage}
\usepackage{graphicx}
\usepackage{multirow}
\usepackage{color}
\usepackage{amsmath}
\usepackage{amsfonts}
\usepackage{amssymb}
\usepackage{mathrsfs}
\usepackage{mathptmx}
\usepackage{eucal,bm}
\usepackage{algorithm}%
\usepackage{algorithmicx}%
\usepackage{algpseudocode}%
\usepackage{smile}
\usepackage{tikz}
\usetikzlibrary{decorations.pathmorphing,patterns}
\usepackage{dutchcal}
\usepackage{appendix}
\usepackage{subfigure}
\usepackage[subfigure]{tocloft}
\allowdisplaybreaks[4]
\usepackage{pdflscape} 
\usepackage{xurl}

\newcommand{\citep}{\cite}

\def\inprod#1#2{\langle#1,#2\rangle}

\usepackage{tcolorbox}

\usepackage{tocloft}
\usepackage[colorlinks=true,linkcolor=red,urlcolor=blue,citecolor=blue]{hyperref}
\hypersetup{breaklinks=true}

\numberwithin{equation}{section}

\newcommand{\beq}{\begin{equation}}
\newcommand{\eeq}{\end{equation}}

\newcommand{\beqnr}{\begin{eqnarray}}
\newcommand{\eeqnr}{\end{eqnarray}}

\newcommand{\benum}{\begin{enumerate}}
\newcommand{\eenum}{\end{enumerate}}

\newtheorem{PB}{Problem}

\newtheorem{THM}{Theorem}
\newtheorem{CO}{Corollary}

\begin{document}

\title{\huge {Acceleration via Perturbations on Low-resolution Ordinary Differential Equations}\thanks{Authors are listed in alphabetical order.}}
\date{}

\author{Xudong Li\thanks{School of Data Science, Fudan University, Shanghai 200433, China ({\tt  lixudong@fudan.edu.cn}). The research of Xudong Li is
supported in part by the National Key R\&D Program of China 2023YFA1009300 and the National Natural Science Foundation of China 12531014 and 12271107.}
~~~~Lei Shi\thanks{School of Mathematical Sciences, Shanghai Key Laboratory for Contemporary
Applied Mathematics and Center for Applied Mathematics, Fudan University, Shanghai 200433, China ({\tt   leishi@fudan.edu.cn}).  The research of Lei Shi is
supported in part by the National Natural Science Foundation of China 12171093. }
~~~~Mingqi Song\thanks{School of Mathematical Sciences, Fudan University, Shanghai 200433, China ({\tt  23110180034@m.fudan.edu.cn}).}
}

\maketitle

\vskip 0.5cm

\begin{abstract}
    Recently, the high-resolution ordinary differential equation (ODE) framework, which retains higher-order terms, has been proposed to analyze gradient-based optimization algorithms. Through this framework, the term $\nabla^2 f(X_t)\dot{X_t}$, known as the gradient-correction term, was found to be essential for reducing oscillations and accelerating the convergence rate of function values. Despite the importance of this term, simply adding it to the low-resolution ODE may sometimes lead to a slower convergence rate.  To fully understand this phenomenon, we propose a generalized perturbed ODE and analyze the role of the gradient and gradient-correction perturbation terms under both continuous-time and discrete-time settings. We demonstrate that while the gradient-correction perturbation is essential for obtaining accelerations, it can hinder the convergence rate of function values in certain cases. However, this adverse effect can be mitigated by involving an additional gradient perturbation term. Moreover, by conducting a comprehensive analysis, we derive proper choices of perturbation parameters. Numerical experiments are also provided to validate our theoretical findings.
\end{abstract}


\section{Introduction}
The swift progression of machine learning contributes to notable advancements in first-order optimization methods. Accelerated first-order methods garner significant attention due to their ability to achieve faster iteration complexity without introducing additional computational overhead compared to their non-accelerated counterparts. A seminal contribution in this domain is Nesterov's accelerated method \cite{Nesterov_1983_method, Nesterov_2018_lectures}. However, the derivations presented therein are often considered counterintuitive and rely heavily on case-specific algebraic manipulations \cite{Juditsky_2013_convex}, thus highlighting the need for a deeper understanding of the acceleration phenomenon.

While there exists a long history linking optimization algorithms with trajectories of ordinary differential equations (ODEs) \cite{Helmke_1996_dynamic, Schropp_2000_system, Fiori_2005_quasi}, it was only recently that Su et al. \cite{Su_2014_differential, Attouch_2022_first-order} effectively connected Nesterov's accelerated scheme for solving smooth convex problems with a specially crafted second-order ODE. Since this groundbreaking work, many subsequent studies \cite{Krichene_2015_accelerated, Wibisono_2016_variational,
Wilson_2021_lyapunov,Yang_2018_physical, Attouch_2022_first-order}
 have endeavored to offer deeper insights and enhanced understanding of the acceleration schemes from the perspective of ODEs. Among these studies, the work \cite{Yang_2018_physical}  drew analogies between the differential equations of some popular algorithms and damped oscillator systems, offering valuable physical insights. 
Quite recently, it was observed in \cite{Shi_2022_understanding} that the continuous ODEs corresponding to the trajectory $X(t)$, derived following the approach in \cite{Su_2014_differential, Wibisono_2016_variational} for two fundamentally different algorithms—Nesterov's accelerated gradient method for $\mu$-strongly convex functions (NAG-SC) and Polyak's heavy-ball method—are identically taking the following form:
\begin{equation}
    \label{low resolution for strong convex}
    \ddot{X_t}+2\sqrt{\mu}\dot{X_t}+\nabla f(X_t)=0,
\end{equation}
where $f(x)$ is a smooth $\mu$-strongly convex function to be minimized and  
the following notation are used: \[X_t=X(t), \quad \dot{X_t}=\frac{\mathrm{d} X_t}{\mathrm{d} t}, \quad \ddot{X_t}=\frac{\mathrm{d}^2 X_t}{\mathrm{d} t^2}.\]
This indicates that the continuous approach promoted in \cite{Su_2014_differential}  may not fully describe the behaviors of discrete accelerated algorithms.
By preserving higher-order terms, in  \cite{Shi_2022_understanding}, the authors  derive the following {\em high-resolution} ODEs 
 \begin{equation}
    \label{high resolution for NAG-SC}
    \ddot{X_t}+2\sqrt{\mu}\dot{X_t}+(1+\sqrt{\mu s})\nabla f(X_t)+\sqrt{s}\nabla^2 f(X_t)\dot{X_t}=0,
\end{equation}
and
\begin{equation}
    \label{high resolution for heavy ball}
    \ddot{X_t}+2\sqrt{\mu}\dot{X_t}+(1+\sqrt{\mu s})\nabla f(X_t)=0,
\end{equation}
where $s$ is the step size in the discrete algorithms, as more accurate surrogates for NAG-SC and the heavy-ball method, respectively. Compared to the low-resolution ODE \eqref{low resolution for strong convex}, the two ODEs \eqref{high resolution for NAG-SC} and \eqref{high resolution for heavy ball} contain extra high order terms $\sqrt{\mu s}\nabla f(X_t)$ and $
\sqrt{s}\nabla^2 f(X_t)\dot{X}_t$, and thus possess more potential in characterizing the performance of NAG-SC and the heavy-ball method. 
As one can observe, the key difference between the high-resolution ODEs of NAG-SC \eqref{high resolution for NAG-SC} and the heavy-ball method \eqref{high resolution for heavy ball} lies in an extra term $\sqrt{s}\nabla^2 f(X_t)\dot{X_t}$, referred to as the gradient correction term, in \eqref{high resolution for NAG-SC}. In \cite{Shi_2022_understanding}, it is emphasized that this term is essential for acceleration. 
Interestingly, an alternative line of research \cite{Alvarez_2000_minimizing, Attouch_2012_second-order, Attouch_2014_dynamical, Attouch_2022_first-order} also highlights the term $\nabla^2 f(X_t)\dot{X_t}$, where it is coined as the Hessian-driven damping term. Unlike the approach in \cite{Shi_2022_understanding}, this line of work derives the term by leveraging second-order information obtained via the Newton method. The pivotal role of the gradient correction term $\nabla^2 f(X_t)\dot{X_t}$ in accelerating optimization algorithms is further underscored in the existing literature.
For instance, this term can effectively neutralize oscillations, as demonstrated in \cite{Attouch_2022_first-order}, and is crucial for achieving a rapid convergence rate of ${\cal o}(1/k^3)$ in the gradient norm minimization of Nesterov's accelerated gradient method for minimizing convex functions (NAG-C) \cite{Chen_2022_gradient}, as well as in the proximal subgradient norm minimization of FISTA \cite{Li_2022_proximal}.

Although the existing literature underscores the crucial impact of the gradient correction term $\nabla^2 f(X_t)\dot{X_t}$, an intriguing anomaly arises wherein the mere inclusion of this term into the low-resolution ODE may paradoxically decrease the convergence rate of the function value, for example the system $(DIN)_{2\sqrt{\mu},\beta}$ in \cite{Attouch_2022_first-order}. Specifically, for any given $\beta \in [0,1/2\sqrt{\mu}]$, $(DIN)_{2\sqrt{\mu},\beta}$ is referred to as an inertial system for minimizing a $\mu$-strongly convex function $f$:
\begin{equation}
    \label{eq: a similar perturbed ODE for strongly convex}
    \ddot{X_t}+2\sqrt{\mu}\dot{X_t}+\beta \nabla^2 f(X_t)\dot{X_t}+\nabla f(X_t)=0.
\end{equation}
The convergence rate of the function value $f(X_t)-f(x^*)$ for \eqref{eq: a similar perturbed ODE for strongly convex} derived in \cite[{Theorem 7(i)}]{Attouch_2022_first-order} is ${\cal O}(e^{-\frac{\sqrt{\mu}}{2}t})$, 
 while the convergence rate of the same quantity for the corresponding low-resolution ODE \eqref{low resolution for strong convex} is a faster decay rate ${\cal O}(e^{-\sqrt{\mu}t})$ \cite[{Proposition 3}]{Wilson_2021_lyapunov}.
In contrast, the high-resolution ODE \eqref{high resolution for NAG-SC} also contains the gradient correction term but exhibits the same convergence rate ${\cal O}(e^{-\sqrt{\mu}t})$ as that of \eqref{low resolution for strong convex}. We also note that the main distinction between \eqref{eq: a similar perturbed ODE for strongly convex} and \eqref{high resolution for NAG-SC} resides in the presence of the perturbation from $\nabla f(X_t)$ within \eqref{high resolution for NAG-SC}. Consequently, the following question naturally arises:
\begin{PB}
    \label{pb: the problems of this paper}
    Does the presence of the gradient correction term $\nabla^2 f(X_t)\dot{X_t}$ adversely affect the convergence rate of the function value? If so, what strategies can be employed to mitigate or eliminate this negative influence?
\end{PB}
In this paper, we address the above problem from the perturbation perspective and propose to study the following perturbed version of \eqref{low resolution for strong convex}:
\begin{equation}
\label{eq:general_perturbed_ODE}
\ddot{X_t}+2\sqrt{\mu}\dot{X_t}+(1+\Delta_1)\nabla f(X_t)+\Delta_2\nabla^2 f(X_t)\dot{X_t}=0,
\end{equation}
where $\Delta_1$, $\Delta_2$ are two nonnegative constants. 
This perturbed ODE extends the system  \eqref{eq: a similar perturbed ODE for strongly convex} and covers both the high-resolution ODEs for NAG-SC \eqref{high resolution for NAG-SC} and the heavy-ball method \eqref{high resolution for heavy ball}.  

We begin by offering intuitive interpretations of the gradient perturbation \( \Delta_1 \nabla f(X_t) \) and the gradient-correction perturbation \( \Delta_2 \nabla^2 f(X_t) \dot{X_t} \) in \eqref{eq:general_perturbed_ODE}.
Specifically, we connect the general perturbed ODE \eqref{eq:general_perturbed_ODE} with a damped oscillator system, a perturbed version of the physical system studied in \cite{Yang_2018_physical}. 
In our model, the term $\Delta_1\nabla f(X_t)$ reinforces the system's resilience, accelerating the particle's return to the equilibrium position, but may increase the oscillations. Meanwhile, the term $\Delta_2\nabla^2 f(X_t)\dot{X_t}$ can be viewed as a force resulting from the change in the impulse $\Delta_2\nabla f(X_t)$. It is negative if the resilience is decreasing. This may slow down the particle's approach to the equilibrium position, but it is beneficial for reducing oscillations. Intuitively, properly combining these two terms may accelerate the convergence. Indeed, we can demonstrate that incorporating only $\Delta_1\nabla f(X_t)$ does not slow down the convergence rate of $f(X_t)-f(x^*)$, specifically, $f(X_t)-f(x^*)={\cal O}(e^{-\sqrt{\mu}t})$.  {In contrast, incorporating only $\Delta_2\nabla ^2 f(X_t)\dot{X_t}$ may decrease the convergence rate.} Moreover, we show that when $\Delta_1$ and $\Delta_2$ are both positive and a proper ratio between them is maintained, the convergence rate of $f(X_t)-f(x^*)$ can even exceed $\cal{O}(e^{-\sqrt{\mu}t})$. 
See Section \ref{sec: continuous case} for more discussions.

Based on the above theoretical advances, we take a step further to study the optimization algorithms, as well as the corresponding perturbation terms, resulting from discretizations of \eqref{eq:general_perturbed_ODE}. For this purpose, we briefly review popular discretizations used in the literature. 
Notably, the Runge-Kutta scheme, the symplectic integration of Hamiltonian systems, the explicit Euler, symplectic Euler, and implicit Euler discretizations have been investigated in  \cite{Zhang_2018_direct, Betancourt_2018_symplectic,  Shi_2019_acceleration, Zhang_2021_revisiting}, respectively. Among these, {symplectic and implicit schemes} exhibit characteristics of simplicity in form, convenience in analysis, and excellent numerical performance. Therefore, we focus on the discrete optimization algorithms obtained by discretizing \eqref{eq:general_perturbed_ODE} using the implicit and symplectic Euler schemes.
Proper conditions on the perturbation parameters $\Delta_1$ and $\Delta_2$  
are proposed to ensure the acceleration of the resulting discrete algorithms, and a new class of accelerated algorithms for minimizing strongly convex functions is derived. We also examine the roles of $\Delta_1\nabla f(x_{k})$ and $\Delta_2(\nabla f(x_{k+1})-\nabla f(x_k))/\sqrt{s}$, which correspond to the discretizations of $\Delta_1\nabla f(X_t)$ and $\Delta_2\nabla^2 f(X_t)\dot{X_t}$. 
For implicit Euler discretization, the two aforementioned terms play roles analogous to their continuous counterparts. However, the situation becomes more intricate for symplectic Euler discretization. We show that in this case, the gradient perturbation $\Delta_1\nabla f(x_k)$ alone is insufficient to ensure a fast convergence rate, and the gradient-correction perturbation $\Delta_2 (\nabla f(x_{k+1})-\nabla f(x_k))/\sqrt{s}$ is crucial for achieving acceleration. Nevertheless, adding only the gradient-correction perturbation may, in some cases—such as when minimizing a strongly convex quadratic function—slow down the convergence rate. The gradient perturbation $\Delta_1 \nabla f(x_k)$ plays a vital role in counteracting this potential drawback.

The main contributions of our paper are summarized below:
\begin{enumerate}
    \item We propose a general perturbed ODE \eqref{eq:general_perturbed_ODE} and analyze the role of the two perturbations $\Delta_1\nabla f(X_t)$ and $\Delta_2\nabla^2 f(X_t)\dot{X_t}$. We highlight that in certain cases, the gradient correction perturbation $\Delta_2\nabla^2 f(X_t)\dot{X_t}$ may negatively impact the convergence rate of function values. A slight involvement of the gradient perturbation $\Delta_1\nabla f(X_t)$ can mitigate this effect.
    \item We study implicit and symplectic Euler discretizations of the perturbed ODE \eqref{eq:general_perturbed_ODE}. A comprehensive analysis of these discretized schemes is conducted, and appropriate choices for the perturbation parameters $\Delta_1$ and $\Delta_2$ are provided.
    
\end{enumerate}

\subsection*{Organization and notations}
We organize the remainder of the paper as follows. In Section \ref{sec: continuous case}, we study the roles of two perturbation terms, $\Delta_1 \nabla f(X_t)$ and $\Delta_2 \nabla^2 f(X_t)\dot{X_t}$, in \eqref{eq:general_perturbed_ODE} from a physical perspective and provide the corresponding proofs. In Section \ref{sec: discrete case}, we analyze the implicit Euler and the discrete symplectic Euler discretization schemes of the perturbed ODE \eqref{eq:general_perturbed_ODE} and discuss the roles of gradient and gradient-correction perturbations.
In Section \ref{sec: numerical experiments}, some preliminary numerical experiments are provided to validate our theoretical results. Lastly, in Section \ref{sec: discussion}, we conclude the paper.

Throughout the paper, we use $\langle \cdot, \cdot \rangle$, and $\|\cdot\|$ to denote the inner product and induced norm in a real finite-dimensional Hilbert space $\cH$, respectively.
We also use $\mathcal{C}^1$ and $\mathcal{C}^2$ to denote the sets of first-order and second-order continuously differentiable functions, respectively.
A function $f\in \mathcal{C}^1$ is said to be $L$-smooth if  $\|\nabla f(x)-\nabla f(y)\|\leqslant L\|x-y\|, \forall x,y \in {\cal H}$, and is said to be $\mu$-strongly convex if  $f(y)-f(x)\geqslant \langle \nabla f(x), y-x \rangle + \mu \|y-x\|^2/2, \forall x,y \in \cH$. 
In this paper, we focus on the following minimization problem 
\begin{equation}\label{basic-min}
	\min \left\{  f(x) \mid \, x\in\cH \right\}, 
\end{equation}
where  $f\in {\cal C}^1$ is assumed to be $\mu$-strongly convex for some $\mu > 0$. Then, the above minimization problem has only one optimal solution, denoted by $x^*$. 


\section{Perturbed ODE for strongly convex functions}
\label{sec: continuous case}
In this section, we focus on the general ODE model \eqref{eq:general_perturbed_ODE}. When the perturbation parameters $\Delta_1=\Delta_2=0$, \eqref{eq:general_perturbed_ODE} reduces to \eqref{low resolution for strong convex}. As is noted in \cite{Yang_2018_physical}, model \eqref{low resolution for strong convex} with the following equivalent form 
\begin{equation}
    \label{eq: motion of low-resolution ODE}
    \ddot{X_t}=-2\sqrt{\mu}\dot{X_t}-\nabla f(X_t)
\end{equation}
describes a damped oscillator system,
where the particle's mass is unitary, the damping coefficient is $2\sqrt{\mu}$, $X_t$ denotes the position of the particle at time $t$, and the function $f$ represents the potential energy. Similarly,  \eqref{eq:general_perturbed_ODE}, in the following equivalent form
\begin{equation}
    \label{equation of motion}
    \ddot{X_t}=-2\sqrt{\mu}\dot{X_t}-\nabla f(X_t)-\Delta_1\nabla f(X_t)-\Delta_2\nabla^2 f(X_t)\dot{X_t},
\end{equation}
describes a perturbed damped oscillator system.
Compared to \eqref{eq: motion of low-resolution ODE}, the above ODE \eqref{equation of motion} includes two additional terms, $\Delta_1\nabla f(X_t)$ and $\Delta_2\nabla^2 f(X_t)\dot{X_t}$. 
Next, we provide an intuitive understanding of these two terms from the physical perspective.

We start by discussing a simple special horizontal damped spring oscillator described by \eqref{equation of motion} with
 $f(X)=\frac{1}{2}\mathcal{K}X^2$. Here, $\mathcal{K}$ is the Hooke's constant of the spring, and $X$ represents the elongation of the spring.  Let $x^*$ denote the position of the spring at its equilibrium length.
In this context, $x^* = 0$ and the two perturbation terms are $\Delta_1\nabla f(X_t)=\Delta_1\mathcal{K}X_t$, $\Delta_2\nabla^2 f(X_t)\dot{X_t}=\Delta_2\mathcal{K}\dot{X_t}$. As shown in Figure \ref{fig:4_stage_motion}, the motion of the object in the system can be divided into four distinct stages.  In the following, we examine the impact of the two perturbation terms on the motion of the object in each of these four stages.

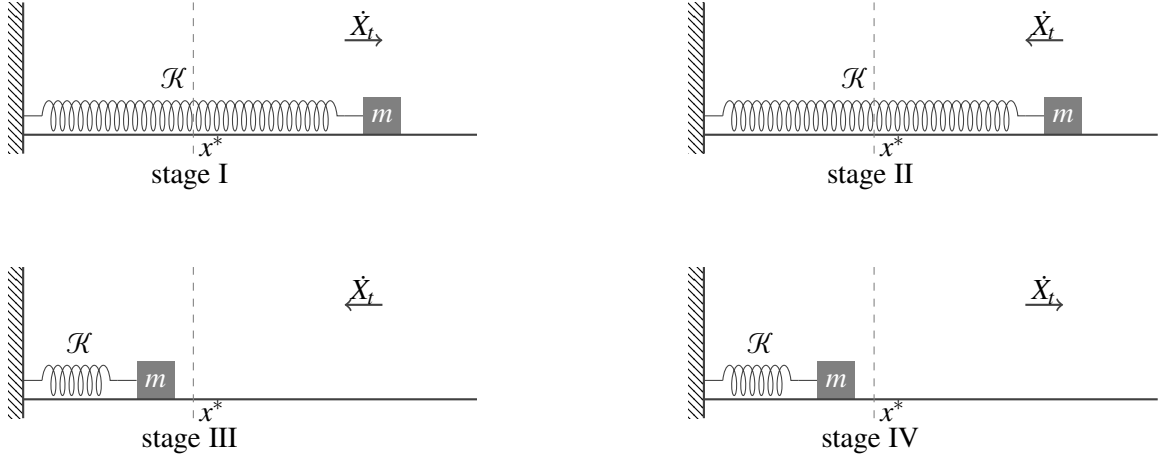
\begin{figure}	
    \begin{tikzpicture}[black!75]
		
		\fill [pattern = north west lines] (-0.2,0) rectangle (0,2);
		\draw[thick] (0,0) -- (0,2);
		
		\draw[] (0,0.5) -- (0.25,0.5);
		\draw[decoration={aspect=0.3, segment length=1.2mm, amplitude=2mm,coil},decorate] (0.25,0.5) -- (4.25,0.5);
        \draw[] (4.25,0.5) -- (4.5,0.5);
        \fill [gray] (4.5,0.25) rectangle (5,0.75);
        \draw[thick] (0,0.25)--(6,0.25);
        \node [white] at (4.75,0.5) {$m$};
        \node[black] at (2,1) {$\mathcal{K}$};
        \draw[dashed,gray] (2.25,0)--(2.25,2);
        \node[black] at (2.5,0.1) {$x^*$};
        \draw[->,thick](4.25,1.5)--(4.75,1.5);
        \node[black] at (4.5,1.7) {$\dot{X_t}$};
        \node[black] at (2.2,-0.3) {stage \uppercase\expandafter{\romannumeral1}};
		
		\begin{scope}[xshift=9cm]
		\fill [pattern = north west lines] (-0.2,0) rectangle (0,2);
		\draw[thick] (0,0) -- (0,2);
		
		\draw[] (0,0.5) -- (0.25,0.5);
		\draw[decoration={aspect=0.3, segment length=1.2mm, amplitude=2mm,coil},decorate] (0.25,0.5) -- (4.25,0.5);
        \draw[] (4.25,0.5) -- (4.5,0.5);
        \fill [gray] (4.5,0.25) rectangle (5,0.75);
        \draw[thick] (0,0.25)--(6,0.25);
        \node [white] at (4.75,0.5) {$m$};
        \node[black] at (2,1) {$\mathcal{K}$};
        \draw[dashed,gray] (2.25,0)--(2.25,2);
        \node[black] at (2.5,0.1) {$x^*$};
        \draw[<-,thick](4.25,1.5)--(4.75,1.5);
        \node[black] at (4.5,1.7) {$\dot{X_t}$};
        \node[black] at (2.2,-0.3) {stage \uppercase\expandafter{\romannumeral2}};
		\end{scope}
		
		\begin{scope}[yshift=-3.5cm]
		\fill [pattern = north west lines] (-0.2,0) rectangle (0,2);
		\draw[thick] (0,0) -- (0,2);
		
		\draw[] (0,0.5) -- (0.25,0.5);
		\draw[decoration={aspect=0.3, segment length=1.2mm, amplitude=2mm,coil},decorate] (0.25,0.5) -- (1.25,0.5);
        \draw[] (1.25,0.5) -- (1.5,0.5);
        \fill [gray] (1.5,0.25) rectangle (2,0.75);
        \draw[thick] (0,0.25)--(6,0.25);
        \node [white] at (1.75,0.5) {$m$};
        \node[black] at (0.75,1) {$\mathcal{K}$};
        \draw[dashed,gray] (2.25,0)--(2.25,2);
        \node[black] at (2.5,0.1) {$x^*$};
		\draw[<-,thick](4.25,1.5)--(4.75,1.5);
        \node[black] at (4.5,1.7) {$\dot{X_t}$};
        \node[black] at (2.2,-0.3) {stage \uppercase\expandafter{\romannumeral3}};
		\end{scope}
		
		\begin{scope}[yshift=-3.5cm,xshift=9cm]
		\fill [pattern = north west lines] (-0.2,0) rectangle (0,2);
		\draw[thick] (0,0) -- (0,2);
		
		\draw[] (0,0.5) -- (0.25,0.5);
		\draw[decoration={aspect=0.3, segment length=1.2mm, amplitude=2mm,coil},decorate] (0.25,0.5) -- (1.25,0.5);
        \draw[] (1.25,0.5) -- (1.5,0.5);
        \fill [gray] (1.5,0.25) rectangle (2,0.75);
        \draw[thick] (0,0.25)--(6,0.25);
        \node [white] at (1.75,0.5) {$m$};
        \node[black] at (0.75,1) {$\mathcal{K}$};
        \draw[dashed,gray] (2.25,0)--(2.25,2);
        \node[black] at (2.5,0.1) {$x^*$};
		\draw[->,thick](4.25,1.5)--(4.75,1.5);
        \node[black] at (4.5,1.7) {$\dot{X_t}$};
        \node[black] at (2.2,-0.3) {stage \uppercase\expandafter{\romannumeral4}};
		\end{scope}
		
	\end{tikzpicture}
\caption{An illustration of four stages of a horizontal damped spring oscillator described by \eqref{equation of motion} with
	$f(X)=\frac{1}{2}\mathcal{K}X^2$.}
    \label{fig:4_stage_motion}
\end{figure}

At \textbf{stage \uppercase\expandafter{\romannumeral1}}, the spring is stretched, and the particle's velocity is directed to the right. The directions of $-\Delta_1\nabla f(X_t)$ and $-\Delta_2\nabla^2 f(X_t)\dot{X_t}$ point to the left. Therefore, at this stage, both perturbation terms accelerate the particle back towards $x^*$. At \textbf{stage \uppercase\expandafter{\romannumeral2}}, the spring remains stretched, but the velocity of the particle is now directed to the left. The directions of $-\Delta_1\nabla f(X_t)$ and $-\Delta_2\nabla^2 f(X_t)\dot{X_t}$ are to the left and right respectively. As a result, $-\Delta_1\nabla f(X_t)$
still accelerates the particle back towards $x^*$, while $\Delta_2\nabla^2 f(X_t)\dot{X_t}$ acts as a ``brake", helping the particle decelerate and stop right at $x^*$. At \textbf{stage \uppercase\expandafter{\romannumeral3}}, the spring is compressed, and the particle's velocity is directed to the left. The directions of $-\Delta_1\nabla f(X_t)$ and $-\Delta_2\nabla^2 f(X_t)\dot{X_t}$ are both directed to the right. Hence, at this stage, the effect of both perturbation terms mirrors that of {\bf stage} \uppercase\expandafter{\romannumeral1}, pushing the particle back towards $x^*$. Finally, at \textbf{stage \uppercase\expandafter{\romannumeral4}}, the spring is compressed, and the particle's velocity is directed to the right. The directions of $-\Delta_1\nabla f(X_t)$ and $-\Delta_2\nabla^2f(X_t)\dot{X_t}$ are to the right and the left, respectively. So, the effects of the two perturbation terms are similar to those in {\bf stage} \uppercase\expandafter{\romannumeral2}.
In summary, the term $-\Delta_2 \nabla^2 f(X_t) \dot{X_t}$ accelerates the particle as it moves away from $x^*$ and decelerates it as it approaches $x^*$, effectively reducing oscillations.
On the other hand, 
regardless of whether the particle is moving away from $x^*$ or towards $x^*$, as long as the object deviates from $x^*$, $-\Delta_1\nabla f(X_t)$ will accelerate its return to $x^*$. Based on these observations, if only $-\Delta_2 \nabla^2 f(X_t) \dot{X_t}$ is present, the deceleration as the particle approaches $x^*$ could slow its convergence rate. Conversely, if only $-\Delta_1 \nabla f(X_t)$ is present, the particle may struggle to stop near $x^*$, resulting in increased oscillations.
Thus, appropriately combining both terms enables the particle to change its moving direction more quickly when traveling away from $x^*$, while also ensuring it reaches $x^*$ at a satisfying speed with minimal oscillations.

For a more general strongly convex potential energy $f$, we can expect similar roles of the two aforementioned perturbation terms. Indeed, as observed in \cite{Yang_2018_physical}, the properties of a strongly convex energy naturally mimic that of a quadratic potential energy.   
Based on this analysis, intuitively, adding appropriate perturbation terms $\Delta_1\nabla f(X_t)$ and $\Delta_2\nabla^2 f(X_t)\dot{X_t}$ will not slow down the system's convergence rate, and may even accelerate it. 
This intuition is proved in the following two results 
based on a Lyapunov analysis with the Lyapunov function $\mathcal{E}$ defined as:
\begin{equation}
	\label{continuous Lyapunov for strong disturb ODE}
	\mathcal{E}(t) :=  e^{\sqrt{\mu}t}\big((1+\Delta_1)(f(X_t)-f(x^*))+\frac{1}{2}\|\dot{X_t}+\sqrt{\mu}(X_t-x^*)+\Delta_2\nabla f(X_t)\|^2\big).
\end{equation}
The following theorem shows that a properly perturbed ODE exhibits the same convergence rate of $f(X_t) - f(x^*)$ as the unperturbed one (i.e. $f(X_t)-f(x^*)={\cal O}(e^{-\sqrt{\mu}t})$ \cite[Proposition 5]{Wilson_2021_lyapunov}).

\begin{THM}
\label{theorem for general ODE}
Suppose that $f\in{\cal C}^2$ is $\mu$-strongly convex. 
Then, the following inequality holds
\begin{equation*}
    \begin{aligned}
        \label{the inequality of dE(t)/dt}
        \frac{\mathrm{d} \mathcal{E}(t)}{\mathrm{d} t}e^{-\sqrt{\mu}t}\leqslant -\frac{\sqrt{\mu}}{2}\|\dot{X_t}\|^2
        -\frac{\mu\sqrt{\mu}}{2}\Delta_1\|X_t-x^*\|^2+\Delta_2(\frac{\sqrt{\mu}}{2}\Delta_2-\Delta_1)\|\nabla f(X_t)\|^2.
    \end{aligned}
\end{equation*}
If the non-negative perturbation parameters $\Delta_1, \Delta_2$ satisfy
\begin{equation}
	\label{continuous condition}
	0 \leqslant \frac{\sqrt{\mu}}{2}\Delta_2\leqslant \Delta_1,
\end{equation}
then it holds that
\begin{equation}
        f(X_t)-f(x^*)\leqslant \frac{1}{1+\Delta_1}e^{-\sqrt{\mu}t}\mathcal{E}(0).
\end{equation}
Besides, if $\Delta_1=0$, $\Delta_2>0$ and $f$ is $L$-smooth, then the following estimation holds
\begin{equation}
    \label{eq: convergence rate of delta1 = 0, delta2 > 0 in continuous case}
    f(X_t)-f(x^*)\leqslant \frac{1}{1+\Delta_1}e^{-\sqrt{\mu}t(1-\Delta_2^2L)}\mathcal{E}(0).
\end{equation}
\end{THM}
\begin{proof}
{By differentiating $\mathcal{E}(t)$ defined in \eqref{continuous Lyapunov for strong disturb ODE} and multiplying both sides by $e^{-\sqrt{\mu}t}$, and recalling \eqref{eq:general_perturbed_ODE}, we have
\begin{equation*}
    \begin{aligned}
        \frac{\mathrm{d}
        \mathcal{E}(t)}{\mathrm{d} t}e^{-\sqrt{\mu}t}
        =&\sqrt{\mu}(1+\Delta_1)\big(f(X_t)-f(x^*)\big)+\frac{\sqrt{\mu}}{2}\|\dot{X_t}+\sqrt{\mu}(X_t-x^*)+\Delta_2\nabla f(X_t)\|^2\\
        &+(1+\Delta_1)\langle \nabla f(X_t),\dot{X_t} \rangle+\langle \dot{X_t}+\sqrt{\mu}(X_t-x^*)+\Delta_2\nabla f(X_t),-\sqrt{\mu}\dot{X_t}-(1+\Delta_1)\nabla f(X_t)\rangle\\
        =&\sqrt{\mu}(1+\Delta_1)\big(f(X_t)-f(x^*)\big)-\frac{\sqrt{\mu}}{2}\|\dot{X_t}\|^2+\frac{\mu\sqrt{\mu}}{2}\|X_t-x^*\|^2
        \\&+\sqrt{\mu}\big(\sqrt{\mu}\Delta_2-(1+\Delta_1)\big)\langle \nabla f(X_t), X_t-x^*\rangle +\Delta_2\big(\frac{\sqrt{\mu}}{2}\Delta_2-(1+\Delta_1)\big)\|\nabla f(X_t)\|^2.
    \end{aligned}
\end{equation*}
}
The $\mu$-strong convexity of $f$ and the optimality of $x^*$ imply that
\begin{equation*}
    \left\{
    \begin{aligned}
        & f(X_t)+\langle \nabla f(X_t), x^*-X_t \rangle + \frac{\mu}{2}\|X_t-x^*\|^2 \leqslant f(x^*),\\
        & \mu \langle \nabla f(X_t), X_t-x^* \rangle  \leqslant \|\nabla f(X_t)\|^2.
    \end{aligned}
    \right. 
\end{equation*}
Therefore, it follows that
\begin{align}
    \frac{\mathrm{d} \mathcal{E}(t)}{\mathrm{d} t}e^{-\sqrt{\mu}t}
        \leqslant & -\frac{\sqrt{\mu}}{2}\|\dot{X_t}\|^2-\frac{\mu\sqrt{\mu}}{2}\Delta_1\|X_t-x^*\|^2+\Delta_2(\frac{\sqrt{\mu}}{2}\Delta_2-\Delta_1)\|\nabla f(X_t)\|^2 \label{eq: proof process of theorem 1}.
\end{align}
By integrating \eqref{eq: proof process of theorem 1}, we see that
\begin{equation*}
    \mathcal{E}(t)+\int_{0}^t e^{\sqrt{\mu}u}\big(\Delta_2(\Delta_1-\frac{\sqrt{\mu}}{2}\Delta_2)\|\nabla f(X_u)\|^2 +\frac{\mu\sqrt{\mu}}{2}\Delta_1\|X_u-x^*\|^2+\frac{\sqrt{\mu}}{2}\|\dot{X_u}\|^2 \big)\,du \leqslant \mathcal{E}(0).
\end{equation*}
Now, if $0 \leqslant \sqrt{\mu}\Delta_2/ 2 \leqslant  \Delta_1$, we know that 
\begin{equation*}
    f(X_t)-f(x^*)\leqslant \frac{1}{1+\Delta_1}e^{-\sqrt{\mu}t}\mathcal{E}(t)\leqslant \frac{1}{1+\Delta_1}e^{-\sqrt{\mu}t}\mathcal{E}(0).
\end{equation*}

Now, if $f$ is $L$-smooth, it holds from the optimality of $x^*$ that 
\begin{equation*}
    \|\nabla f(X_t)\|^2 \leqslant 2L\big(f(X_t)-f(x^*)\big).
\end{equation*}
Therefore, \eqref{eq: proof process of theorem 1}  with $\Delta_1=0$ further implies that 
\begin{align*}
    \frac{\mathrm{d} \mathcal{E}(t)}{\mathrm{d} t}e^{-\sqrt{\mu}t}\leqslant & -\frac{\sqrt{\mu}}{2}\|\dot{X_t}\|^2+\frac{\sqrt{\mu}}{2}\Delta_2^2\|\nabla f(X_t)\|^2\\
    \leqslant & \sqrt{\mu}\Delta_2^2L \big(f(X_t)-f(x^*)\big)\\
    \leqslant & \sqrt{\mu}\Delta_2^2 L e^{-\sqrt{\mu}t}\mathcal{E}(t),
\end{align*}
where the last inequality holds from the definition of $\mathcal{E}(t)$ in \eqref{continuous Lyapunov for strong disturb ODE}.
Thus, we have ${\cal E}(t) \le {\cal E}(0) e^{\sqrt{\mu}\Delta_2^2 Lt}$, which further implies that 
\begin{align*}
    f(X_t)-f(x^*)\leqslant e^{-\sqrt{\mu}t}\mathcal{E}(t)\leqslant e^{-\sqrt{\mu}t(1-\Delta_2^2L)} \mathcal{E}(0).
\end{align*}
This completes the proof of the theorem.
\end{proof}
\begin{remark}
In \cite{Attouch_2022_first-order}, the authors discussed a special case of the general mode \eqref{eq:general_perturbed_ODE} with $\Delta_1 = 0$, i.e.,
	\begin{equation}
		\label{first order ODE}\ddot{X_t}+2\sqrt{\mu}\dot{X_t}+\nabla f(X_t)+\Delta_2\nabla^2 f(X_t)\dot{X_t}=0.
	\end{equation}
As is shown in \cite[Theorem 7(i)]{Attouch_2022_first-order}, if $0\leqslant \Delta_2 \leqslant {1}/{(2\sqrt{\mu})}$, then it holds that  $f(X_t)-f(x^*)={\cal O}(e^{-\frac{1}{2}\sqrt{\mu}t})$, which is slower than ${\cal O}(e^{-\sqrt{\mu}t})$, the convergence rate resulted from the unperturbed ODE \eqref{low resolution for strong convex}. This observation aligns with \eqref{eq: convergence rate of delta1 = 0, delta2 > 0 in continuous case} indicating that adding only the perturbation term $\Delta_2\nabla^2 f(X_t)\dot{X_t}$ may slow the convergence rate of $f(X_t)-f(x^*)$. 

On the contrary,  Theorem \ref{theorem for general ODE} shows that involving only the perturbation term $\Delta_1 \nabla f(X_t)$, i.e., $\Delta_2 = 0$ and $\Delta_1 \geqslant 0$, in \eqref{eq:general_perturbed_ODE}, will not slow the convergence rate of $f(X_t)-f(x^*)$. However, as we will show later (see Section \ref{subsec: direct symplectic discretization}), this is not the case for the symplectic Euler discretization case.

\end{remark}

Next, we show that with a slightly strict assumption on $\Delta_1$ and $\Delta_2$, i.e., \eqref{faster continuous condition}, it is possible to obtain an even faster convergence rate.

\begin{THM}
\label{faster convergence rate for general ODE}
Suppose that $f\in{\cal C}^2$ is $\mu$-strongly convex.
If 
\begin{equation*}
    \label{faster continuous condition}
    0<\frac{\sqrt{\mu}}{2}\Delta_2< \Delta_1, 
\end{equation*}
then it holds that
\begin{equation*}  
 \frac{\mathrm{d}\mathcal{E}(t)}{\mathrm{d}t} \leqslant -c_1\mathcal{E}(t) \quad \mbox{ with } \quad c_1=\min \left\{ \frac{2\mu\Delta_2}{1+\Delta_1+3\mu\Delta_2^2}(\Delta_1-\frac{\sqrt{\mu}}{2}\Delta_2),\frac{\sqrt{\mu}}{3},\frac{\sqrt{\mu}}{3}\Delta_1  \right\}>0.
\end{equation*}
Thus, 
\begin{equation*}
        f(X_t)-f(x^*)\leqslant \frac{1}{1+\Delta_1}e^{-(\sqrt{\mu}+c_1)t}\mathcal{E}(0).
\end{equation*}
\end{THM}
\begin{proof}
From \eqref{continuous Lyapunov for strong disturb ODE}, we see that
\begin{equation*}
    \begin{aligned}
        \mathcal{E}(t)e^{-\sqrt{\mu}t}
        =&(1+\Delta_1)\big(f(X_t)-f(x^*)\big)+\frac{1}{2}\|\dot{X_t}+\sqrt{\mu}(X_t-x^*)+\Delta_2\nabla f(X_t)\|^2\\
        \leqslant &(1+\Delta_1)\big(f(X_t)-f(x^*)\big)+\frac{3}{2}\|\dot{X_t}\|^2+\frac{3\mu}{2}\|X_t-x^*\|^2+\frac{3}{2}\Delta_2^2\|\nabla f(X_t)\|^2\\
        \leqslant &\frac{1+\Delta_1+3\mu \Delta_2^2}{2\mu} \|\nabla f(X_t)\|^2+\frac{3}{2}\|\dot{X_t}\|^2+\frac{3\mu}{2}\|X_t-x^*\|^2,
    \end{aligned}
\end{equation*}
{where the first inequality is due to the Cauchy-Schwarz inequality
\begin{equation*}
    \|\dot{X_t}+\sqrt{\mu}(X_t-x^*)+\Delta_2\nabla f(X_t)\|^2 \leqslant 3\big(\|\dot{X_t}\|^2+\mu\|X_t-x^*\|^2+\Delta_2^2\|\nabla f(X_t)\|^2\big),
\end{equation*}
and the second inequality follows from the $\mu$-strong convexity of $f$ and the optimality of $x^*$
\begin{equation*}
    f(X_t)-f(x^*)\leqslant \frac{1}{2\mu}\|\nabla f(X_t)\|^2.
\end{equation*}}
Then, from the definition of $c_1$, we further have
\begin{equation*}\label{eq:c1et}
c_1 {\cal E}(t)e^{-\sqrt{\mu}t} \le \Delta_2(\Delta_1-\frac{\sqrt{\mu}}{2}\Delta_2)\|\nabla f(X_t)\|^2+\frac{\sqrt{\mu}}{2}\|\dot{X_t}\|^2+\frac{\mu\sqrt{\mu}}{2}\Delta_1\|X_t-x^*\|^2.
\end{equation*}
This, together with \eqref{eq: proof process of theorem 1}, implies that
\[
\frac{d \mathcal{E}(t)}{d t}e^{-\sqrt{\mu}t} \leqslant -c_1\mathcal{E}(t)e^{-\sqrt{\mu}t}.
\]
Solving the above ODE inequality and recalling \eqref{continuous Lyapunov for strong disturb ODE}, we have
$$e^{c_1 t}\mathcal{E}(t)=e^{(c_1+\sqrt{\mu})t}\big((1+\Delta_1)(f(X_t)-f(x^*))+\frac{1}{2}\|\dot{X_t}+\sqrt{\mu}(X_t-x^*)+\Delta_2\nabla f(X_t)\|^2\big)\leqslant \mathcal{E}(0).$$
Thus,
\begin{equation*}
    f(X_t)-f(x^*)\leqslant \frac{1}{1+\Delta_1}e^{-(\sqrt{\mu}+c_1)t}\mathcal{E}(0).
\end{equation*}
This completes the proof of the theorem.
\end{proof}

\section{Optimization algorithms obtained by discretizing \texorpdfstring{\eqref{eq:general_perturbed_ODE}}{(general perturbed ODE)}}
\label{sec: discrete case}

The previous section shows that under proper choices of $\Delta_1$ and $\Delta_2$, the resulting trajectory of \eqref{eq:general_perturbed_ODE} enjoys a favorable convergence rate of function values.  In this section, we demonstrate that a proper time discretization of the perturbed dynamic \eqref{eq:general_perturbed_ODE}, combined with carefully chosen values of $\Delta_1$ and $\Delta_2$, yields first-order optimization algorithms with fast convergence properties.
For this purpose, we focus on the following phase-space form of the perturbed ODE \eqref{eq:general_perturbed_ODE}
\begin{equation}
	\label{general perturbation ODE phase representation}
	\left\{
	\begin{aligned}
		&\frac{\mathrm{d} X}{\mathrm{d} t}=\dot{X},\\
		&\frac{\mathrm{d} \dot{X}}{\mathrm{d} t}=-2\sqrt{\mu}\dot{X}-(1+\Delta_1)\nabla f(X)-\Delta_2\nabla^2 f(X)\dot{X}.
	\end{aligned}
	\right.
\end{equation}
The above reformulation is closely related to the phase-space representation technique proposed in \cite{Shi_2022_understanding}, which has yielded interesting results \cite{Chen_2022_gradient, Li_2022_proximal,  Li_2024_linear, Shi_2019_acceleration} in accelerated algorithms.
In this section, we study optimization algorithms by taking the popular implicit Euler and symplectic Euler discretizations on \eqref{general perturbation ODE phase representation}.
The implicit and symplectic Euler schemes are well-known discretizations for solving ODEs, and have recently been highlighted in the study of accelerated optimization algorithms. See, for example, \cite{Betancourt_2018_symplectic, Jordan_2018_dynamical, Shi_2019_acceleration, França_2021_dissipative}. In particular, the discrete algorithms obtained by the implicit and symplectic discretizations of the phase-space form of the unperturbed ODE \eqref{low resolution for strong convex}, i.e. $\Delta_1 = \Delta_2 = 0$ in \eqref{general perturbation ODE phase representation}, have been investigated in \cite{Shi_2019_acceleration}. 

Here, for \eqref{general perturbation ODE phase representation}, we utilize Lyapunov functions translated from the continuous case via the phase-space representation to show that appropriate perturbations do not slow down and can even accelerate the convergence of function values. Furthermore, our analysis leads to new accelerated methods that extend the acceleration techniques proposed in \cite{Shi_2019_acceleration}.

\subsection{Optimization algorithms obtained by the implicit discretization}
\label{subsec: implicit discretization}
We start by discretizing \eqref{general perturbation ODE phase representation} 
using the following implicit Euler scheme:
\begin{equation}
    \label{strong disturb implicit iteration}
    \left\{
    \begin{aligned}
        &\frac{x_{k+1}-x_k}{\sqrt{s}}=v_{k+1},\\
        &\frac{v_{k+1}-v_k}{\sqrt{s}}=-2\sqrt{\mu}v_{k+1}-(1+\Delta_1)\nabla f(x_{k+1})-\Delta_2\frac{\nabla f(x_{k+1})-\nabla f(x_k)}{\sqrt{s}}.
    \end{aligned}
    \right.
\end{equation} 
Associated with  \eqref{strong disturb implicit iteration}, similar to the continuous case in \eqref{continuous Lyapunov for strong disturb ODE}, we define the following Lyapunov function:
\begin{equation}
    \label{Lyapunov for strong disturb implicit iteration}
    E(k)=(1+\sqrt{\mu s})^k\big((1+\Delta_1)(f(x_k)-f(x^*))+\frac{1}{2}\|v_{k}+\sqrt{\mu}(x_k-x^*)+\Delta_2\nabla f(x_k)\|^2\big).
\end{equation}
With this potential function, we derive the convergence rate of $f(x_k) - f^*$ in the following theorem. 
\begin{THM}
\label{general implicit corollary}
Suppose that $f\in {\mathcal C}^1$ is  $\mu$-strongly convex.
If the non-negative perturbation parameters $\Delta_1, \Delta_2$ satisfy
\begin{equation}
    \label{condition for implicit}
    0\leqslant \frac{\sqrt{\mu}}{2}\Delta_2\leqslant \Delta_1,
\end{equation}
then for any step size $s>0$ and any initial point $x_0$ and $v_0$,  
it holds that
\begin{equation}\label{eq:rate_con_simp}
    f(x_k)-f(x^*)\leqslant \frac{1}{1+\Delta_1}(1+\sqrt{\mu s})^{-k}E(0), \quad \forall\, k \geqslant 0.
\end{equation}
\end{THM}
\begin{proof}
{
Recalling the definition of $E(k)$ in \eqref{Lyapunov for strong disturb implicit iteration}, we see that
\begin{equation}\label{eq:diffEk}
    \begin{aligned}
    &(1+\sqrt{\mu s})^{-k}\big(E(k+1)-E(k) \big)\\
    =&(1+\Delta_1)\big(f(x_{k+1})-f(x_k)\big)+\sqrt{\mu s}(1+\Delta_1)\big(f(x_{k+1})-f(x^*)\big) + M_1^k + M_2^k,
\end{aligned}
\end{equation}
where 
\[M_1^k: = \frac{1}{2}\|v_{k+1}+\sqrt{\mu}(x_{k+1}-x^*)+\Delta_2\nabla f(x_{k+1})\|^2-\frac{1}{2}\|v_k+\sqrt{\mu}(x_k-x^*)+\Delta_2\nabla f(x_k)\|^2\]
and \[M_2^k := \frac{\sqrt{\mu s}}{2}\|v_{k+1}+\sqrt{\mu}(x_{k+1}-x^*)+\Delta_2\nabla f(x_{k+1})\|^2.\]
Based on the iterative scheme \eqref{strong disturb implicit iteration}, we have that
\begin{equation*}
    v_{k+1}-v_k+\sqrt{\mu}(x_{k+1}-x_k)+\Delta_2(\nabla f(x_{k+1})-\nabla f(x_k))=-\sqrt{\mu s}v_{k+1}-(1+\Delta_1)\sqrt{s}\nabla f(x_{k+1}),
\end{equation*}
which further implies that
\begin{align*}
    M_1^k
    =&\langle v_{k+1}+\sqrt{\mu}(x_{k+1}-x^*)+\Delta_2 \nabla f(x_{k+1}), -\sqrt{\mu s}v_{k+1}-(1+\Delta_1)\sqrt{s}\nabla f(x_{k+1}) \rangle\\
   & -\frac{1}{2}\|\sqrt{\mu s}v_{k+1}+(1+\Delta_1)\sqrt{s}\nabla f(x_{k+1}) \|^2.
\end{align*}
Then, we have by simple calculations that
\begin{equation}
    \label{eq:m1pm2}
    \begin{aligned}
    M_1^k + M_2^k = & -\frac{\sqrt{\mu s}}{2}\|v_{k+1}\|^2
            +\frac{\mu}{2}\sqrt{\mu s}\|x_{k+1}-x^*\|^2+\Delta_2\sqrt{s}\big(\frac{\sqrt{\mu}}{2}\Delta_2 -(1+\Delta_1)\big)\|\nabla f(x_{k+1})\|^2\\
            &-(1+\Delta_1)\langle \nabla f(x_{k+1}),x_{k+1}-x_{k}\rangle +\sqrt{\mu s}\big(\sqrt{\mu}\Delta_2-(1+\Delta_1)\big)\langle x_{k+1}-x^*, \nabla f(x_{k+1})\rangle \\
            & -\frac{1}{2}\|\sqrt{\mu s}v_{k+1}+(1+\Delta_1)\sqrt{s}\nabla f(x_{k+1})\|^2.
    \end{aligned}
\end{equation}
Now, from \eqref{eq:diffEk} and \eqref{eq:m1pm2}, we have that 
\begin{align*}
    &(1+\sqrt{\mu s})^{-k}\big(E(k+1)-E(k)\big)\\
    =&(1+\Delta_1)\big(f(x_{k+1})-f(x_k)\big)+\sqrt{\mu s}(1+\Delta_1)\big(f(x_{k+1})-f(x^*)\big)-\frac{\sqrt{\mu s}}{2}\|v_{k+1}\|^2\\
            &+\frac{\mu}{2}\sqrt{\mu s}\|x_{k+1}-x^*\|^2+\Delta_2\sqrt{s}\big(\frac{\sqrt{\mu}}{2}\Delta_2-(1+\Delta_1)\big)\|\nabla f(x_{k+1})\|^2-(1+\Delta_1)\langle \nabla f(x_{k+1}),x_{k+1}-x_{k}\rangle\\
            &+\sqrt{\mu s}\big(\sqrt{\mu}\Delta_2-(1+\Delta_1)\big)\langle x_{k+1}-x^*, \nabla f(x_{k+1})\rangle -\frac{1}{2}\|\sqrt{\mu s}v_{k+1}+(1+\Delta_1)\sqrt{s}\nabla f(x_{k+1})\|^2 \\
    =&(1+\Delta_1)\big(f(x_{k+1})-f(x_k)+\langle \nabla f(x_{k+1}),x_k-x_{k+1} \rangle \big) + \mu\sqrt{s}\Delta_2\langle x_{k+1}-x^*,\nabla f(x_{k+1})\rangle  \\
    & +\sqrt{\mu s}(1+\Delta_1)\big(f(x_{k+1})-f(x^*)+\langle \nabla f(x_{k+1}), x^*-x_{k+1}\rangle \big) +\frac{\mu}{2}\sqrt{\mu s}\|x_{k+1}-x^*\|^2 \\
    & +\Delta_2\sqrt{s}\big(\frac{\sqrt{\mu}}{2}\Delta_2-(1+\Delta_1)\big)\|\nabla f(x_{k+1})\|^2 - \frac{\sqrt{\mu s}}{2}\|v_{k+1}\|^2-\frac{1}{2}\|\sqrt{\mu s}v_{k+1}+(1+\Delta_1)\sqrt{s}\nabla f(x_{k+1})\|^2.
\end{align*}
}
{Since $f$ is $\mu$-strongly convex and $x^*$ is the optimal solution, it holds that}
    \begin{equation*}
        \left\{
        \begin{aligned}
            & f(x_{k+1})-f(x^*)+\langle x^*-x_{k+1},\nabla f(x_{k+1})\rangle+\frac{\mu}{2}\|x_{k+1}-x^*\|^2\leqslant 0,\\
            & f(x_{k+1})-f(x_k)+\langle x_k-x_{k+1},\nabla f(x_{k+1})\rangle +\frac{\mu}{2}\|x_{k+1}-x_k\|^2\leqslant 0,\\
            & \mu \langle x_{k+1}-x^*, \nabla f(x_{k+1}) \rangle \leqslant \|\nabla f(x_{k+1})\|^2.
        \end{aligned}
        \right.
    \end{equation*}
   {Therefore,  we have
    \begin{equation*}
        \begin{aligned}
            &(1+\sqrt{\mu s})^{-k}\big(E(k+1)-E(k)\big)\\
            \leqslant{} & -\frac{\mu}{2}(1+\Delta_1)\|x_{k+1}-x_k\|^2-\frac{\mu\sqrt{\mu s}}{2}\Delta_1\|x_{k+1}-x^*\|^2+\Delta_2\sqrt{s}\big(\frac{\sqrt{\mu}}{2}\Delta_2-\Delta_1\big)\|\nabla f(x_{k+1})\|^2\\
            &-\frac{\sqrt{\mu s}}{2}\|v_{k+1}\|^2-\frac{1}{2}\|\sqrt{\mu s}v_{k+1}+(1+\Delta_1)\sqrt{s}\nabla f(x_{k+1})\|^2\\
            \leqslant {}&0,
        \end{aligned}
    \end{equation*}
    where the last inequality holds under the condition \eqref{condition for implicit}.
    }
    Hence, 
    \begin{equation*}
        E(k+1) \leqslant E(k),\quad \forall\, k \geqslant 0.
    \end{equation*}
    Thus, using \eqref{Lyapunov for strong disturb implicit iteration} and the above inequality, we obtain
    \begin{equation*}
        f(x_k)-f(x^*)\leqslant \frac{1}{1+\Delta_1}(1+\sqrt{\mu s})^{-k}E(k)\leqslant \frac{1}{1+\Delta_1}(1+\sqrt{\mu s})^{-k}E(0).
    \end{equation*}
    This completes the proof of the theorem.
\end{proof}

    The condition \eqref{condition for implicit} is identical to \eqref{continuous condition} in Theorem \ref{theorem for general ODE}. Set $k\sqrt{s}\equiv t$, and let $s\to 0+$, then the discrete Lyapunov function  $E(k)$ converges to $ \mathcal{E}(t)$ defined in  \eqref{continuous Lyapunov for strong disturb ODE}, and the rate $(1+\sqrt{\mu s})^{-k}$ converges to $ e^{-\sqrt{\mu}t}$. Therefore, Theorem \ref{general implicit corollary} can be considered a discrete counterpart of Theorem \ref{theorem for general ODE}. This implies that the implicit Euler discretization \eqref{strong disturb implicit iteration} effectively preserves the convergence properties of the trajectory of the continuous system \eqref{eq:general_perturbed_ODE}.

In \cite[Theorem 3.2(c)]{Shi_2019_acceleration}, the algorithm obtained by the implicit discretization of the phase-space ODE \eqref{general perturbation ODE phase representation} for the low-resolution ODE \eqref{low resolution for strong convex}, i.e. $\Delta_1=0$, $\Delta_2=0$, has been shown to possess a convergence rate 
\[f(x_k)-f(x^*)={\cal O}\big((1+\frac{1}{4}\sqrt{\mu s})^{-k}\big)\]
 for a $\mu$-strongly convex, $L$-smooth function $f$ if the step size $s$ satisfies $0<s\leqslant 1/L$. 
Here, under a weaker assumption, i.e., only assuming $f$ to be $\mu$-strongly convex, Theorem \ref{general implicit corollary} obtains a stronger and broader result, i.e., when  $0\leqslant \Delta_2 \sqrt{\mu}/2 \leqslant \Delta_1$, 
\[f(x_k)-f(x^*)={\cal O}\big((1+\sqrt{\mu s})^{-k}\big)\] for any step size $s>0$.
Thus, similar to the continuous case, Theorem \ref{general implicit corollary} shows that proper perturbation will not slow down the convergence rate of $f(x_k) - f(x^*)$ compared to the unperturbed case. In particular, involving only the perturbation term $\Delta_1 \nabla f(x_{k+1})$ in \eqref{strong disturb implicit iteration}, i.e., $\Delta_2 =0$ and $\Delta_1 \ge 0$, will not slow down the convergence rate.

We shall also mention that by simple calculations,
the iterative scheme \eqref{strong disturb implicit iteration} can be rewritten into the following form 
\begin{equation}
    \label{alg: implicit discretization}
    \left\{
    \begin{aligned}
        & y_k =x_k+\frac{\Delta_2\sqrt{s}}{1+2\sqrt{\mu s}}\nabla f(x_k)-\frac{1}{1+2\sqrt{\mu s}}(x_k-x_{k-1}),\\
     & x_{k+1} =\mathrm{prox}_{\beta f}(y_k) \,  \mbox{ with } \,  \beta=\frac{\sqrt{s}}{1+2\sqrt{\mu s}}[(1+\Delta_1)\sqrt{s}+\Delta_2].
    \end{aligned}
    \right.
\end{equation}
 Note that, in \eqref{alg: implicit discretization}, the proximal mapping associated with $f$ is included and thus poses computational difficulties in practical applications. 

\subsection{Optimization algorithms obtained by symplectic discretizations}
\label{subsec: direct symplectic discretization}
In this subsection, we first discretize the phase space ODE \eqref{general perturbation ODE phase representation} using the symplectic Euler scheme and arrive at the following updating formula:
\begin{equation}
    \label{strong disturb symplectic iteration}
    \left\{
    \begin{aligned}
        &\frac{x_{k+1}-x_k}{\sqrt{s}}=v_{k},\\
        &\frac{v_{k+1}-v_k}{\sqrt{s}}=-2\sqrt{\mu}v_{k+1}-(1+\Delta_1)\nabla f(x_{k+1})-\Delta_2\frac{\nabla f(x_{k+1})-\nabla f(x_k)}{\sqrt{s}},
    \end{aligned}
    \right.
\end{equation}
which can be further equivalently rewritten as:
\begin{equation}
    \label{eq: iteration of direct symplectic discretization}
    x_{k+1}=x_k+\frac{1}{1+2\sqrt{\mu s}}(x_k-x_{k-1})-\frac{1+\Delta_1}{1+2\sqrt{\mu s}}s\nabla f(x_k)-\frac{\Delta_2\sqrt{s}}{1+2\sqrt{\mu s}}\big(\nabla f(x_k)-\nabla f(x_{k-1})\big).
\end{equation}
The following Lyapunov function is used to analyze the convergence rate of \eqref{eq: iteration of direct symplectic discretization}:
\begin{equation}
    \label{Lyapunov for symplectic 2}
    \begin{aligned}
        E(k)=&\Big(1+\frac{\sqrt{\mu s}}{1+\sqrt{\mu s}}\Big)^k \big((1+\Delta_1)\big(f(x_k)-f(x^*)-\frac{\Delta_2\sqrt{s}}{2}\|\nabla f(x_k)\|^2\big)\\
        &+\frac{1}{2}\|v_k+\sqrt{\mu}(x_{k+1}-x^*)+\Delta_2\nabla f(x_k)\|^2\big).
    \end{aligned}
\end{equation}
Unlike the one used in \eqref{Lyapunov for strong disturb implicit iteration}, the Lyapunov function $E$ in \eqref{Lyapunov for symplectic 2} contains an extra term $-\Delta_2\sqrt{s}\|\nabla f(x_k)\|^2/2$ in the first part associated with the difference of function value $f(x_k) - f(x^*)$.
\begin{THM}
    \label{th:symplectic 2}
    Suppose that $f$ is $\mu$-strongly convex and $L$-smooth.  
    If the non-negative perturbation parameters $\Delta_1, \Delta_2$, and step size $s>0$ satisfy:
\begin{equation}
    \label{conditions of direct symplectic 2}
    \begin{aligned}
        (1)& \, \Delta_2\sqrt{s}\leqslant \frac{1}{L};\\
        (2)& \, \Delta_2\leqslant \sqrt{s}(1+\Delta_1);\\
        (3)& \, \frac{\sqrt{\mu s}}{1+\sqrt{\mu s}}\Delta_2^2-\Delta_2\sqrt{s}(1+\Delta_1)\Big(\frac{\sqrt{\mu s}}{1+\sqrt{\mu s}}+2\Big)+(1+\Delta_1)^2s
        -\frac{\sqrt{\mu s}}{1+\sqrt{\mu s}}\frac{\Delta_1}{L}\\
        &+\frac{2\mu\sqrt{s}}{(1+\sqrt{\mu s})L}\big(\Delta_2-\sqrt{s}(1+\Delta_1)\big)\leqslant 0,
    \end{aligned}
\end{equation}
then for any initial point $x_0$ and $v_0$, the sequence $\{x^k\}$ generated by \eqref{eq: iteration of direct symplectic discretization} satisfies that 
\begin{equation*}
    f(x_k)-f(x^*)-\frac{\Delta_2\sqrt{s}}{2}\|\nabla f(x_k)\|^2 \leqslant \frac{1}{1+\Delta_1}\Big(1+\frac{\sqrt{\mu s}}{1+\sqrt{\mu s}}\Big)^{-k}E(0).
\end{equation*}
Thus, if in addition $\Delta_2\sqrt{s}< 1/L$, then
\begin{equation}\label{eq:f_rate_symplectic}
    f(x_k)-f(x^*)\leqslant \frac{1}{(1-L\Delta_2\sqrt{s})(1+\Delta_1)}\Big(1+\frac{\sqrt{\mu s}}{1+\sqrt{\mu s}}\Big)^{-k}E(0).
\end{equation}
\end{THM}
\begin{proof}
The proof for the current theorem is quite similar to the one for Theorem  \ref{general implicit corollary}. Particularly, we will argue that $E(k)$, defined in  \eqref{Lyapunov for symplectic 2}, is nonincreasing across the iteration $k$. For this purpose, we compute
    \begin{equation}
        \label{eq: diffEk in direct symplectic discretization}
        \begin{aligned}
            &\Big(1+\frac{\sqrt{\mu s}}{1+\sqrt{\mu s}}\Big)^{-k} \big( E(k+1)-E(k)\big)\\
            =&(1+\Delta_1)\big(f(x_{k+1})-f(x_k)\big)+\frac{\sqrt{\mu s}}{1+\sqrt{\mu s}}(1+\Delta_1)\big(f(x_{k+1})-f(x^*)\big)+M_1^k+M_2^k\\
            &-\frac{\Delta_2\sqrt{s}}{2}(1+\Delta_1)(\|\nabla f(x_{k+1})\|^2-\|\nabla f(x_k)\|^2)-\frac{\Delta_2\sqrt{\mu}s}{2(1+\sqrt{\mu s})}(1+\Delta_1)\|\nabla f(x_{k+1})\|^2,
        \end{aligned}
    \end{equation}
    where
    \begin{equation*}
        M_1^k := \frac{1}{2}\|v_{k+1}+\sqrt{\mu}(x_{k+2}-x^*)+\Delta_2\nabla f(x_{k+1})\|^2-\frac{1}{2}\|v_k+\sqrt{\mu}(x_{k+1}-x^*)+\Delta_2\nabla f(x_k)\|^2,
    \end{equation*}
    and
    \begin{equation*}
        M_2^k := \frac{\sqrt{\mu s}}{2(1+\sqrt{\mu s})}\|v_{k+1}+\sqrt{\mu}(x_{k+2}-x^*)+\Delta_2\nabla f(x_{k+1})\|^2.
    \end{equation*} 
    Simplifying $M_1^k$ using \eqref{strong disturb symplectic iteration}, we obtain
    \begin{align*}
    M_1^k
        =&\langle v_{k+1}+\sqrt{\mu}(x_{k+2}-x^*)+\Delta_2\nabla f(x_{k+1}), -\sqrt{\mu s}v_{k+1}-(1+\Delta_1)\sqrt{s}\nabla f(x_{k+1}) \rangle\\
            &-\frac{1}{2}\|\sqrt{\mu s}v_{k+1}+(1+\Delta_1)\sqrt{s}\nabla f(x_{k+1})\|^2.
    \end{align*}
    Since $x_{k+2}=x_{k+1}+\sqrt{s}v_{k+1}$, we have
    \begin{align*}
        M_1^k+M_2^k=&\langle (1+\sqrt{\mu s})v_{k+1}+\sqrt{\mu}(x_{k+1}-x^*)+\Delta_2\nabla f(x_{k+1}), -\sqrt{\mu s}v_{k+1}-(1+\Delta_1)\sqrt{s}\nabla f(x_{k+1}) \rangle\\
            &-\frac{1}{2}\|\sqrt{\mu s}v_{k+1}+(1+\Delta_1)\sqrt{s}\nabla f(x_{k+1})\|^2\\
            &+\frac{\sqrt{\mu s}}{2(1+\sqrt{\mu s})}\|(1+\sqrt{\mu s})v_{k+1}+\sqrt{\mu}(x_{k+1}-x^*)+\Delta_2\nabla f(x_{k+1})\|^2\\
            =
            &-\frac{\sqrt{\mu s}}{2}(1+2\sqrt{\mu s})\|v_{k+1}\|^2+\frac{\mu\sqrt{\mu s}}{2(1+\sqrt{\mu s})}\|x_{k+1}-x^*\|^2\\
            &+\big(\frac{\sqrt{\mu s}}{2(1+\sqrt{\mu s})}\Delta_2^2-\Delta_2\sqrt{s}(1+\Delta_1)-\frac{1}{2}(1+\Delta_1)^2s\big)\|\nabla f(x_{k+1})\|^2\\
            &-(1+\Delta_1)\sqrt{s}(1+2\sqrt{\mu s})\langle v_{k+1},\nabla f(x_{k+1})\rangle\\
            &+\big(\frac{\sqrt{\mu s}}{1+\sqrt{\mu s}}\sqrt{\mu}\Delta_2-\sqrt{\mu s}(1+\Delta_1)\big)\langle x_{k+1}-x^*, \nabla f(x_{k+1}) \rangle.
    \end{align*}
    Using \eqref{strong disturb symplectic iteration}, we have 
    \begin{equation*}
    (1+2\sqrt{\mu s})v_{k+1}=v_k-(1+\Delta_1)\sqrt{s}\nabla f(x_{k+1})-\Delta_2\big(\nabla f(x_{k+1})-\nabla f(x_k)\big),
    \end{equation*}
    and thus
    \begin{align*}
        &(1+\Delta_1)\sqrt{s}(1+2\sqrt{\mu s})\langle v_{k+1}, \nabla f(x_{k+1}) \rangle\\
        =&(1+\Delta_1)\langle x_{k+1}-x_k, \nabla f(x_{k+1})\rangle -(1+\Delta_1)^2s\|\nabla f(x_{k+1})\|^2\\
        &-(1+\Delta_1)\sqrt{s}\Delta_2\langle \nabla f(x_{k+1})-\nabla f(x_k), \nabla f(x_{k+1})\rangle\\
        =&(1+\Delta_1)\langle x_{k+1}-x_k, \nabla f(x_{k+1})\rangle -(1+\Delta_1)^2s\|\nabla f(x_{k+1})\|^2\\
        &-\frac{1}{2}(1+\Delta_1)\sqrt{s}\Delta_2(\|\nabla f(x_{k+1})\|^2-\|\nabla f(x_k)\|^2+\|\nabla f(x_{k+1})-\nabla f(x_k)\|^2),
    \end{align*}
    where the first equality is due to $\sqrt{s}v_k=x_{k+1}-x_k$, the second equality follows from
    \begin{align*}
        \langle \nabla f(x_{k+1})-\nabla f(x_k), \nabla f(x_{k+1})\rangle=\frac{1}{2}(\|\nabla f(x_{k+1})\|^2-\|\nabla f(x_k)\|^2+\|\nabla f(x_{k+1})-\nabla f(x_k)\|^2).
    \end{align*}
    Then, it holds that
    \begin{equation}
        \label{eq: M1M2 in direct symplectic discretization}
        \begin{aligned}
            M_1^k+M_2^k=&-\frac{\sqrt{\mu s}}{2}(1+2\sqrt{\mu s})\|v_{k+1}\|^2+\frac{\mu \sqrt{\mu s}}{2(1+\sqrt{\mu s})}\|x_{k+1}-x^*\|^2\\
        &+\big(\frac{\sqrt{\mu s}}{2(1+\sqrt{\mu s})}\Delta_2^2-\Delta_2\sqrt{s}(1+\Delta_1)+\frac{1}{2}(1+\Delta_1)^2s\big)\|\nabla f(x_{k+1})\|^2\\
        &-(1+\Delta_1)\langle x_{k+1}-x_k,\nabla f(x_{k+1})\rangle +\big(\frac{\sqrt{\mu s}}{1+\sqrt{\mu s}}\sqrt{\mu}\Delta_2-\sqrt{\mu s}(1+\Delta_1)\big)\langle x_{k+1}-x^*, \nabla f(x_{k+1})\rangle\\
        &+\frac{1}{2}(1+\Delta_1)\Delta_2\sqrt{s}(\|\nabla f(x_{k+1})\|^2-\|\nabla f(x_k)\|^2+\|\nabla f(x_{k+1})-\nabla f(x_k)\|^2).
        \end{aligned}
    \end{equation}
    Now, we combine \eqref{eq: diffEk in direct symplectic discretization} and \eqref{eq: M1M2 in direct symplectic discretization} and obtain
    \begin{align*}
        &\Big(1+\frac{\sqrt{\mu s}}{1+\sqrt{\mu s}}\Big)^{-k}\big[E(k+1)-E(k)\big]\\
        =&(1+\Delta_1)\big(f(x_{k+1})-f(x_k)\big)+\frac{\sqrt{\mu s}}{1+\sqrt{\mu s}}(1+\Delta_1)\big(f(x_{k+1})-f(x^*)\big)\\
        &-\frac{\Delta_2\sqrt{s}}{2}(1+\Delta_1)(\|\nabla f(x_{k+1})\|^2-\|\nabla f(x_k)\|^2)-\frac{\Delta_2\sqrt{\mu}s}{2(1+\sqrt{\mu s})}(1+\Delta_1)\|\nabla f(x_{k+1})\|^2\\
        &-\frac{\sqrt{\mu s}}{2}(1+2\sqrt{\mu s})\|v_{k+1}\|^2+\frac{\mu \sqrt{\mu s}}{2(1+\sqrt{\mu s})}\|x_{k+1}-x^*\|^2\\
        &+\big(\frac{\sqrt{\mu s}}{2(1+\sqrt{\mu s})}\Delta_2^2-\Delta_2\sqrt{s}(1+\Delta_1)+\frac{1}{2}(1+\Delta_1)^2s\big)\|\nabla f(x_{k+1})\|^2\\
        &-(1+\Delta_1)\langle x_{k+1}-x_k,\nabla f(x_{k+1})\rangle +\big(\frac{\sqrt{\mu s}}{1+\sqrt{\mu s}}\sqrt{\mu}\Delta_2-\sqrt{\mu s}(1+\Delta_1)\big)\langle x_{k+1}-x^*, \nabla f(x_{k+1})\rangle\\
        &+\frac{1}{2}(1+\Delta_1)\Delta_2\sqrt{s}(\|\nabla f(x_{k+1})\|^2-\|\nabla f(x_k)\|^2+\|\nabla f(x_{k+1})-\nabla f(x_k)\|^2)\\
        =&(1+\Delta_1)\big(f(x_{k+1})-f(x_k)+\langle \nabla f(x_{k+1}), x_k-x_{k+1} \rangle + \frac{\Delta_2\sqrt{s}}{2}\|\nabla f(x_{k+1})-\nabla f(x_k)\|^2\big)\\
        &+\frac{\sqrt{\mu s}}{1+\sqrt{\mu s}}\big(f(x_{k+1})-f(x^*)+\langle x^*-x_{k+1},\nabla f(x_{k+1})\rangle+\frac{\mu}{2}\|x_{k+1}-x^*\|^2 \big)\\
        &+\frac{\sqrt{\mu s}}{1+\sqrt{\mu s}}\Delta_1 \big( f(x_{k+1})-f(x^*)+\langle x^*-x_{k+1},\nabla f(x_{k+1})\rangle \big)\\
        &-\frac{\sqrt{\mu s}}{2}(1+2\sqrt{\mu s})\|v_{k+1}\|^2 + \frac{\sqrt{\mu s}}{1+\sqrt{\mu s}}\big(\sqrt{\mu}\Delta_2-\sqrt{\mu s}(1+\Delta_1)\big)\langle x_{k+1}-x^*, \nabla f(x_{k+1})\rangle\\
        &+\frac{1}{2}\|\nabla f(x_{k+1})\|^2\big(-\frac{\Delta_2\sqrt{\mu}s}{1+\sqrt{\mu s}}(1+\Delta_1)+\frac{\sqrt{\mu s}}{1+\sqrt{\mu s}}\Delta_2^2-2\Delta_2\sqrt{s}(1+\Delta_1)^2+(1+\Delta_1)^2s\big).
    \end{align*}
By the $\mu$-strong convexity and $L$-smoothness of $f$, we obtain
    \begin{equation*}
        \left\{
        \begin{aligned}
            &f(x_{k+1})+\langle \nabla f(x_{k+1}),x_k-x_{k+1} \rangle +\frac{1}{2L}\|\nabla f(x_{k+1})-\nabla f(x_k)\|^2 \leqslant f(x_k),\\
            &f(x_{k+1})+\langle \nabla f(x_{k+1}),x^*-x_{k+1}\rangle +\frac{1}{2L}\|\nabla f(x_{k+1})\|^2 \leqslant f(x^*),\\
            &f(x_{k+1})+\langle \nabla f(x_{k+1}),x^*-x_{k+1}\rangle +\frac{\mu}{2}\|x_{k+1}-x^*\|^2 \leqslant
            f(x^*).\\
            &\langle x_{k+1}-x^*, \nabla f(x_{k+1}) \rangle \geqslant \frac{1}{L}\|\nabla f(x_{k+1})\|^2.
        \end{aligned}
        \right.
    \end{equation*}
The above inequalities further imply that 
    \begin{align*}
        &\Big(1+\frac{\sqrt{\mu s}}{1+\sqrt{\mu s}}\Big)^{-k}\big[E(k+1)-E(k)\big]\\
        \leqslant&\frac{1+\Delta_1}{2}(\Delta_2\sqrt{s}-\frac{1}{L})\|\nabla f(x_{k+1})-\nabla f(x_k)\|^2 -\frac{\sqrt{\mu s}}{2}(1+2\sqrt{\mu s})\|v_{k+1}\|^2\\
        &+\frac{1}{2}\|\nabla f(x_{k+1})\|^2
        \Big[-\frac{\Delta_2\sqrt{\mu}s}{1+\sqrt{\mu s}}(1+\Delta_1)+\frac{\sqrt{\mu s}}{1+\sqrt{\mu s}}\Delta_2^2-2\Delta_2\sqrt{s}(1+\Delta_1)\\
        &+(1+\Delta_1)^2s-\frac{\sqrt{\mu s}}{1+\sqrt{\mu s}}\frac{\Delta_1}{L}+\frac{2\sqrt{\mu s}}{(1+\sqrt{\mu s})L}\big(\sqrt{\mu}\Delta_2-\sqrt{\mu s}(1+\Delta_1)\big)\Big]\\
        \leqslant & 0,
    \end{align*}
where the last inequality follows from \eqref{conditions of direct symplectic 2}. 
Then, it holds by recalling the definition of $E(k)$ in \eqref{Lyapunov for symplectic 2} that
$$f(x_{k})-f(x^*)-\frac{\Delta_2\sqrt{s}}{2}\|\nabla f(x_k)\|^2\leqslant \frac{1}{1+\Delta_1}\Big(1+\frac{\sqrt{\mu s}}{1+\sqrt{\mu s}}\Big)^{-k}E(k) \leqslant \frac{1}{1+\Delta_1}\Big(1+\frac{\sqrt{\mu s}}{1+\sqrt{\mu s}}\Big)^{-k}E(0).$$

{Next, we prove \eqref{eq:f_rate_symplectic}. Since $f$ is $L$-smooth and $x^*$ is the optimal solution, we have
\begin{equation*}
    \|\nabla f(x_{k})\|^2 \leqslant 2L \big(f(x_k)-f(x^*)\big),
\end{equation*}
which, together with the above inequality and the condition that $\Delta_2\sqrt{s}< 1/L$, implies that
\begin{align*}
    f(x_k)-f(x^*) \leqslant & \frac{1}{1-L\Delta_2\sqrt{s}}\big(f(x_k)-f(x^*)-\frac{\Delta_2\sqrt{s}}{2}\|\nabla f(x_k)\|^2 \big)\\ \leqslant & \frac{1}{(1-L\Delta_2\sqrt{s})(1+\Delta_1)}\Big(1+\frac{\sqrt{\mu s}}{1+\sqrt{\mu s}}\Big)^{-k}E(0).
\end{align*}
}
This completes the proof of the theorem.
\end{proof}

In the above theorem, condition \eqref{conditions of direct symplectic 2} seems to be complicated. By simple calculations, we can reformulate (3) in \eqref{conditions of direct symplectic 2} to be:
\begin{equation*}
    \begin{aligned}
        &\big(\Delta_2-\frac{\sqrt{s}}{2}(1+\Delta_1)\big)\big(\frac{\sqrt{\mu s}}{1+\sqrt{\mu s}}\big(\Delta_2-\frac{\sqrt{s}}{2}(1+\Delta_1)\big)-2\sqrt{s}(1+\Delta_1)\big)-\frac{s}{4}(1+\Delta_1)^2\\
        &\qquad +\frac{\sqrt{\mu s}}{1+\sqrt{\mu s}}\frac{2\sqrt{\mu}}{L}\big(\Delta_2-\sqrt{s}(1+\Delta_1)\big)-\frac{\sqrt{\mu s}}{1+\sqrt{\mu s}}\Delta_1 \leqslant 0.
    \end{aligned}
\end{equation*}
Thus, we can replace (2), (3) in \eqref{conditions of direct symplectic 2} by the following simple sufficient condition
$$\frac{\sqrt{s}}{2}(1+\Delta_1)\leqslant \Delta_2 \leqslant \sqrt{s}(1+\Delta_1).$$
As a result, the following corollary can be readily obtained.
\begin{CO}
    \label{symplectic 2 sufficient condition}
     Suppose that $f$ is $\mu$-strongly and $L$-smooth. If the non-negative perturbation parameters $\Delta_1, \Delta_2$, and step size $s>0$ satisfy the conditions:
\begin{equation}
    \label{conditions for symplectic 2 sufficient}
    \begin{aligned}
        (1)& \,  \Delta_2\sqrt{s}< \frac{1}{L};\\
        (2)& \, \frac{\sqrt{s}}{2}(1+\Delta_1)\leqslant \Delta_2 \leqslant \sqrt{s}(1+\Delta_1),
    \end{aligned}
\end{equation}
then for any initial points $x_0$, $v_0$, it holds that 
\begin{equation}
	\label{eq:ratecorollary}
	f(x_k)-f(x^*)\leqslant \frac{1}{(1-L\Delta_2\sqrt{s})(1+\Delta_1)}\Big(1+\frac{\sqrt{\mu s}}{1+\sqrt{\mu s}}\Big)^{-k}E(0).
\end{equation}
\end{CO}

Here, we shall compare the rate result in \eqref{eq:ratecorollary} with that in \cite{Shi_2019_acceleration}. Under the same setting as in \cite[Theorem 3.1(a)]{Shi_2019_acceleration}, i.e., $\Delta_1=\sqrt{\mu s}$, $\Delta_2=\sqrt{s}$, and $s=4/(9L)$, the result in \eqref{eq:ratecorollary} yields \[f(x_k)-f(x^*)={\cal O}\Big( \big(1+\frac{2\sqrt{\mu/L}}{3+2\sqrt{\mu/L}})^{-k}\Big)={\cal O}\big( (1+\frac{2}{5} \sqrt{\mu/L})^{-k}\big),\]
which improve the result $f(x_k)-f(x^*)={\cal O}\big((1+1/9\sqrt{\mu/L})^{-k}\big)$ in \cite[Theorem 3.1(a)]{Shi_2019_acceleration}. This comparison demonstrates that proper perturbations could further accelerate the convergence rate. In fact, we can obtain a class of accelerated algorithms with an even larger step-size, i.e., $s = 1/L$ and a better convergence rate. 
We summarize the corresponding algorithm in Algorithm \ref{algorithm:symplectic discretization 1}. It can be regarded as a generalization of the one obtained in \cite{Shi_2019_acceleration} by applying the symplectic scheme on the high-resolution system for NAG-SC \eqref{high resolution for NAG-SC}.
Under the following condition on the perturbation parameters $\Delta_1$ and $\Delta_2$  
\[(1+\Delta_1)/2 \le \sqrt{L}\Delta_2 < 1,\]
one can show by \eqref{eq:ratecorollary} that the sequence generated by Algorithm \ref{algorithm:symplectic discretization 1} satisfies
\[f(x_k)-f(x^*)={\cal O}\big((1+\frac{\sqrt{\mu/L}}{1+\sqrt{\mu/L}})^{-k}\big)={\cal O}\Big((1+\frac{1}{2}\sqrt{\mu/L})^{-k}\Big).\]

\begin{algorithm}
\caption{Optimization algorithm derived from the symplectic Euler discretization}
\label{algorithm:symplectic discretization 1}
\begin{algorithmic}[1]
\State{Choose $x_0 \in \cH$ and  $\Delta_1, \Delta_2 \ge 0$ satisfying $(1+\Delta_1)/2\leqslant \sqrt{L}\Delta_2<1$, and set $x_1=x_0$.}
\For{$k=1,2,\ldots$}
\State{$x_{k+1}=x_k+\frac{1}{1+2\sqrt{\mu/L}}(x_k-x_{k-1})-\frac{1+\Delta_1}{(1+2\sqrt{\mu/L})L}\nabla f(x_k)-\frac{\Delta_2}{(1+2\sqrt{\mu/L})\sqrt{L}}\big(\nabla f(x_k)-\nabla f(x_{k-1})\big)$.}
\EndFor
\end{algorithmic}
\end{algorithm}

Next, we examine the special choices of perturbation parameters $\Delta_1$ and $\Delta_2$, and discuss relations with known results. We start with the unperturbed case where $\Delta_1 = \Delta_2=0$. Then, 
\eqref{conditions of direct symplectic 2} requires that
\begin{equation}
    \label{eq: conditions for direct symplectic when delta1=delta2=0}
    s\leqslant \frac{2\mu s}{(1+\sqrt{\mu s})L} \, \mbox{ or equivalently } \, \sqrt{\mu s} \leqslant \frac{2\mu}{L} - 1.
\end{equation}
Therefore, in this case, the desirable step size $s$ may not exist when $\mu/L \leqslant 1/2$. This indicates that the symplectic discretization of the low-resolution ODE \eqref{low resolution for strong convex} may be difficult to obtain an accelerated convergence rate. 
Indeed, it has been shown in \cite{Shi_2019_acceleration} that the algorithm obtained by the symplectic discretization of the phase-space form of the low-resolution ODE \eqref{low resolution for strong convex} enjoys the convergence rate \[f(x_k)-f(x^*)={\cal O}\big((1+\frac{1}{4}\sqrt{\mu s})^{-k}\big)\] for the step size $s$ satisfying $0<s\leqslant \mu/(16L^2)$, thus not achieving acceleration. Next, we consider the case where $\Delta_1 = 0$ and $\Delta_2 >0$. In this case, as discussed before, by setting $s = 1/L$ and $1/2\sqrt{L} \leqslant \Delta_2 < 1/\sqrt{L}$, one can obtain a class of algorithms with the following accelerated convergence rate 
\[f(x_k) - f(x^*) = {\cal O}\Big((1+\frac{1}{2}\sqrt{\mu/L})^{-k}\Big). \]
Lastly, we consider the case where $\Delta_1 >0$ and $\Delta_2 = 0$. In this case, 
\eqref{conditions of direct symplectic 2} becomes:
\begin{equation}
    \label{eq: condition of direct symplectic discretization when delta2=0}
    s\leqslant \frac{2\mu s}{(1+\sqrt{\mu s})L}\frac{1}{1+\Delta_1}+\frac{\sqrt{\mu s}}{(1+\sqrt{\mu s})L}\frac{\Delta_1}{(1+\Delta_1)^2}.
\end{equation}
Simple calculations assert that all the possible step sizes $s$ satisfying \eqref{eq: condition of direct symplectic discretization when delta2=0} are of the order ${\cal O}(\mu/L^2)$, which is the same order obtained in  \cite[Theorem 3.2(a)]{Shi_2019_acceleration}. Then, \eqref{eq:f_rate_symplectic} in Theorem \ref{th:symplectic 2} implies the non-accelerated rate
$
f(x^k) - f(x^*) = {\cal O}\big( (1 + \frac{1}{2} \mu/L)^{-k}\big),
$
coinciding with the conclusion in \cite[Theorem 3.2(a)]{Shi_2019_acceleration}.

The above discussions highlight that
unlike the continuous case and the case of implicit discretization, with only the perturbation term $\Delta_1\nabla f(x_k)$ in \eqref{strong disturb symplectic iteration} using the symplecitc discretization (i.e., $\Delta_1 >0$ and $\Delta_2 =0$), the sequence generated by \eqref{eq: iteration of direct symplectic discretization} fails to achieve acceleration. Meanwhile, our findings also align with previous work indicating that $\Delta_2 >0$ is crucial for enabling large step sizes and desired accelerated convergence rates of $f(x_k)-f(x^*)$. Similar to our physical interpretation on the gradient perturbation term $-\Delta_2 \nabla^2 f(X_t) \dot{X_t}$, we observe here that for the symplectic discretization scheme \eqref{strong disturb symplectic iteration}, involving only the gradient perturbation term $\Delta_2\big(\nabla f(x_{k+1})-\nabla f(x_k)\big)/\sqrt{s}$ (i.e., $\Delta_1 =0$ and $\Delta_2 >0$ in \eqref{strong disturb symplectic iteration}) may not be a good choice. In fact, as is shown in the following example on minimizing a toy convex quadratic function, the convergence rate corresponding to the case $\Delta_1 =0$ and $\Delta_2 >0$  can be slower than the unperturbed case $\Delta_1 = \Delta_2 = 0$, while a proper choice of $\Delta_1 >0$ and $\Delta_2 >0$ could result better convergence rate. For this purpose, we define $A = {\rm Diag}([\mu; L])$ with given $0 < \mu < L$ and consider 
\[ \min_{x\in\mathbb{R}^2} f(x): = \frac{1}{2} \inprod{x}{Ax}. \]
For this special problem, the symplectic discretization scheme \eqref{eq: iteration of direct symplectic discretization} generates the sequence $\{x_k = [z_k^1; z^2_k]\}$ in the following way:
\begin{equation*}
    \label{eq: one dimension iteration}
    z_{k+1}^i=\big(1-\frac{\lambda_i s (1+\Delta_1)}{1+2\sqrt{\mu s}}+\frac{1-\lambda_i\sqrt{s} \Delta_2}{1+2\sqrt{\mu s}}\big)z_k^i+\frac{\lambda_i \sqrt{s}\Delta_2-1}{1+2\sqrt{\mu s}}z_{k-1}^i,
\end{equation*}
where $i=1,2$, and $\lambda_1 = \mu$ and $\lambda_2 = L$. In the following analysis, the step-size $s$ is set to be $s = 1/L$. For the case $\Delta_1 =0$ and $\Delta_2 = 1/\sqrt{L}$, a tight analysis reveals that 
\[
f(x^k) = {\cal O}\big ((\frac{1+\sqrt{\mu/L}}{1+2\sqrt{\mu/L}} )^{2k}\big ),
\] 
which is slower than that for the unperturbed case (i.e., $\Delta_1 = \Delta_2 = 0$) with
\[
f(x^k) = {\cal O}\big ((\frac{1}{1+2\sqrt{\mu/L}} )^{k}\big ).
\] 
Now, if we set $\Delta_1 = \sqrt{\mu/L}$ and $\Delta_2 = 1/\sqrt{L}$, the convergence rate can be further improved to at least
\[
f(x^k) = {\cal O}\Big (\big(\frac{1+2\sqrt{\mu/L}-3\mu^2/(4L^2)}{(1+2\sqrt{\mu/L})^2} \big)^{k}\Big ).
\]
These findings will be further illustrated through numerical experiments in the next section.

We further note that in the updating scheme \eqref{eq: iteration of direct symplectic discretization}, the momentum coefficient is 
$1/(1+2\sqrt{\mu s})$ rather than  $(1-\sqrt{\mu s})/(1+\sqrt{\mu s})$ as in the classic NAG-SC. To obtain the same momentum coefficient as NAG-SC, we propose the following modified symplectic discretization 
\begin{equation}
    \label{strong disturb symplectic 2 iteration}
    \left\{
    \begin{aligned}
        &\frac{x_{k+1}-x_k}{\sqrt{s}}=v_k,\\
        &\frac{v_{k+1}-v_k}{\sqrt{s}}=-\sqrt{\mu}(v_{k+1}+v_k)-(1+\Delta_1)\nabla f(x_{k+1})-\Delta_2\frac{\nabla f(x_{k+1})-\nabla f(x_k)}{\sqrt{s}},
    \end{aligned}
    \right.
\end{equation}
where the key difference between \eqref{strong disturb symplectic 2 iteration} and \eqref{strong disturb symplectic iteration} is the use of $-\sqrt{\mu}(v_{k+1}+v_k)$ rather than $-2\sqrt{\mu}v_{k+1}$ as the discretization of $-2\sqrt{\mu}\dot{X_t}$.
Note that \eqref{strong disturb symplectic 2 iteration} can be rewritten as
\begin{equation}
    \label{eq: iteration of modified symplectic discretization}
    x_{k+1}=x_k+\frac{1-\sqrt{\mu s}}{1+\sqrt{\mu s}}(x_k-x_{k-1})-\frac{1+\Delta_1}{1+\sqrt{\mu s}}s\nabla f(x_{k})-\frac{\Delta_2}{1-\sqrt{\mu s}}\sqrt{s}\big(\nabla f(x_{k})-\nabla f(x_{k-1})\big).
\end{equation}
As promised, in \eqref{eq: iteration of modified symplectic discretization}, we have $(1-\sqrt{\mu s})/(1+\sqrt{\mu s})$ as the momentum coefficient. The analysis for the scheme \eqref{eq: iteration of modified symplectic discretization} is documented in Appendix \ref{sec:appdenx_modified}.

\section{Numerical experiments}
\label{sec: numerical experiments}

In this section, we conduct numerical experiments to verify our theoretical findings. Our primary focus is on the numerical experiments on the direct symplectic discretization scheme \eqref{eq: iteration of direct symplectic discretization}. 
We test the updating scheme \eqref{eq: iteration of direct symplectic discretization} with different perturbation parameters $\Delta_1$ and $\Delta_2$ on minimizing a $\mu$-strongly convex and $L$-smooth function $f$. Specifically, these two perturbation parameters are chosen from the following four cases:
\[
(\Delta_1, \Delta_2) \in \{(0,0), (\widehat{\Delta}_1, 0), (\widehat{\Delta}_2, 0), (\widehat{\Delta}_1, \widehat{\Delta}_2) \}
\] 
with some given $\widehat{\Delta}_1$ and $\widehat{\Delta}_2$. In our tests, we set $s = 1/L$, and choose $\widehat{\Delta_1} \in \{\sqrt{\mu s}, 1 \}$ and $\widehat{\Delta}_2 \in \{ \sqrt{s}, 2\sqrt{s}/3\}$. For comparison, we also run the classic NAG-SC as the baseline algorithm. In the tests, the accuracy of an approximate solution $\tilde x \in \cH$ is measured by $\eta = \|\nabla f(\tilde x)\|$, and the tested algorithms will be terminated if $\eta < \epsilon$ with $\epsilon >0$ being a given stopping tolerance.
Here, we set $\epsilon = 10^{-6}$. The experiments are conducted by running Matlab (version {9.12}) on a notebook ({4}-core, {Intel(R) Core(TM) i5-8250U@1.60GHz, 8 Gigabytes of RAM}).

\subsection{Numerical experiments on convex quadratic programming}
\label{subsec: numerical experiments on convex quadratic programming problem}
We first test the algorithms on the following simple two-dimensional convex quadratic programming problem
\begin{equation}
    \label{eq: quadratic function for numerical experiments}
    \min_{x\in \mathbb{R}^2} f(x)=\frac{1}{2}x^TAx,
\end{equation}
where $A={\rm Diag}([\mu;L])$ with $\mu=1$, $L=100$, as it has been analyzed in Section \ref{subsec: direct symplectic discretization} and thus serves as a clear illustrative example.
For our scheme \eqref{eq: iteration of direct symplectic discretization}, we initialize  
\[
x_0 = [1;1], \quad x_1 = x_0 - \frac{s(1+\Delta_1)}{1+2\sqrt{\mu s}} \nabla f(x_0).
\]  
Figure \ref{pic: all numerical results of xk of direct symplectic discretization} presents a detailed comparison between tested schemes, where the logarithm of the function value difference, \(\log_{10}(f(x_k) - f(x^*))\), is plotted against the iteration count \(k\).  

\begin{figure}
    \centering
    \subfigure[$\hat{\Delta}_1=\sqrt{\mu s}$, $\hat{\Delta}_2=\sqrt{s}$]{
    \label{pic: numerical results of xk of direct symplectic discretization}
    \includegraphics[width=0.49\linewidth]{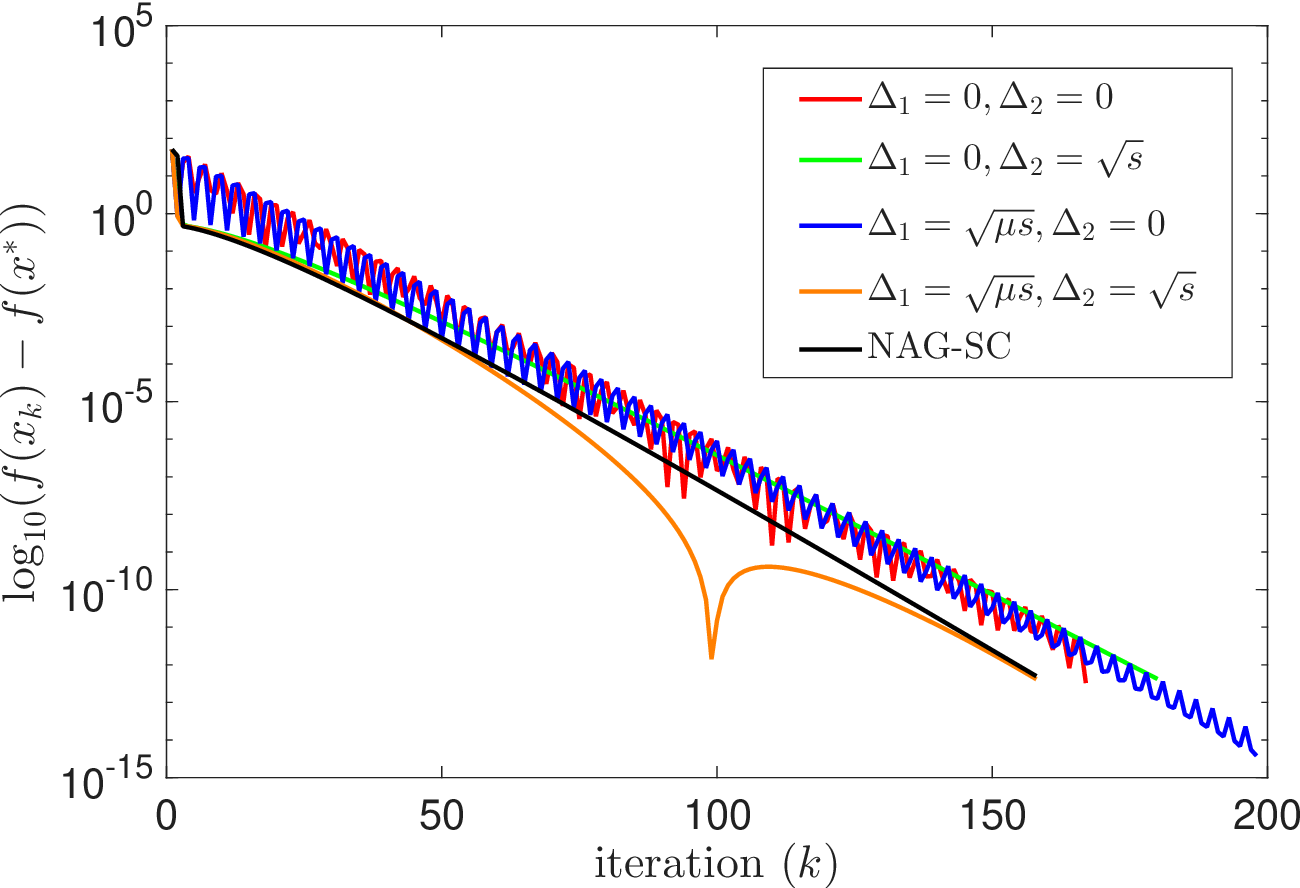}}
    \subfigure[$\hat{\Delta}_1=1$, $\hat{\Delta}_2=\sqrt{s}$]{
    \label{pic: numerical results of xk of direct symplectic discretization for additional delta1}
    \includegraphics[width=0.49\linewidth]{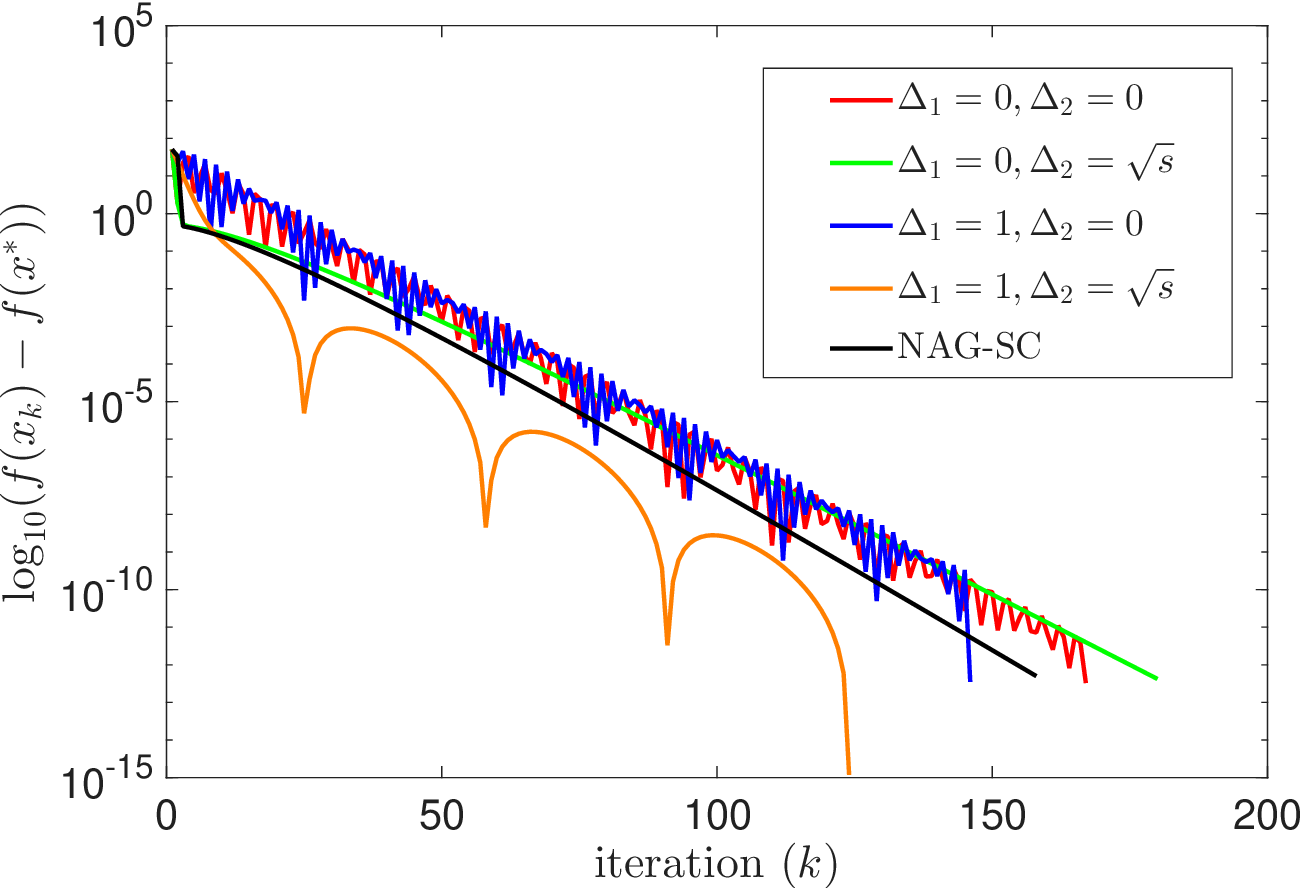}}
    \subfigure[$\hat{\Delta}_1=\sqrt{\mu s}$, $\hat{\Delta}_2=2\sqrt{s}/3$]{
    \label{pic: numerical results of xk of direct symplectic discretization for additional delta2}
    \includegraphics[width=0.49\linewidth]{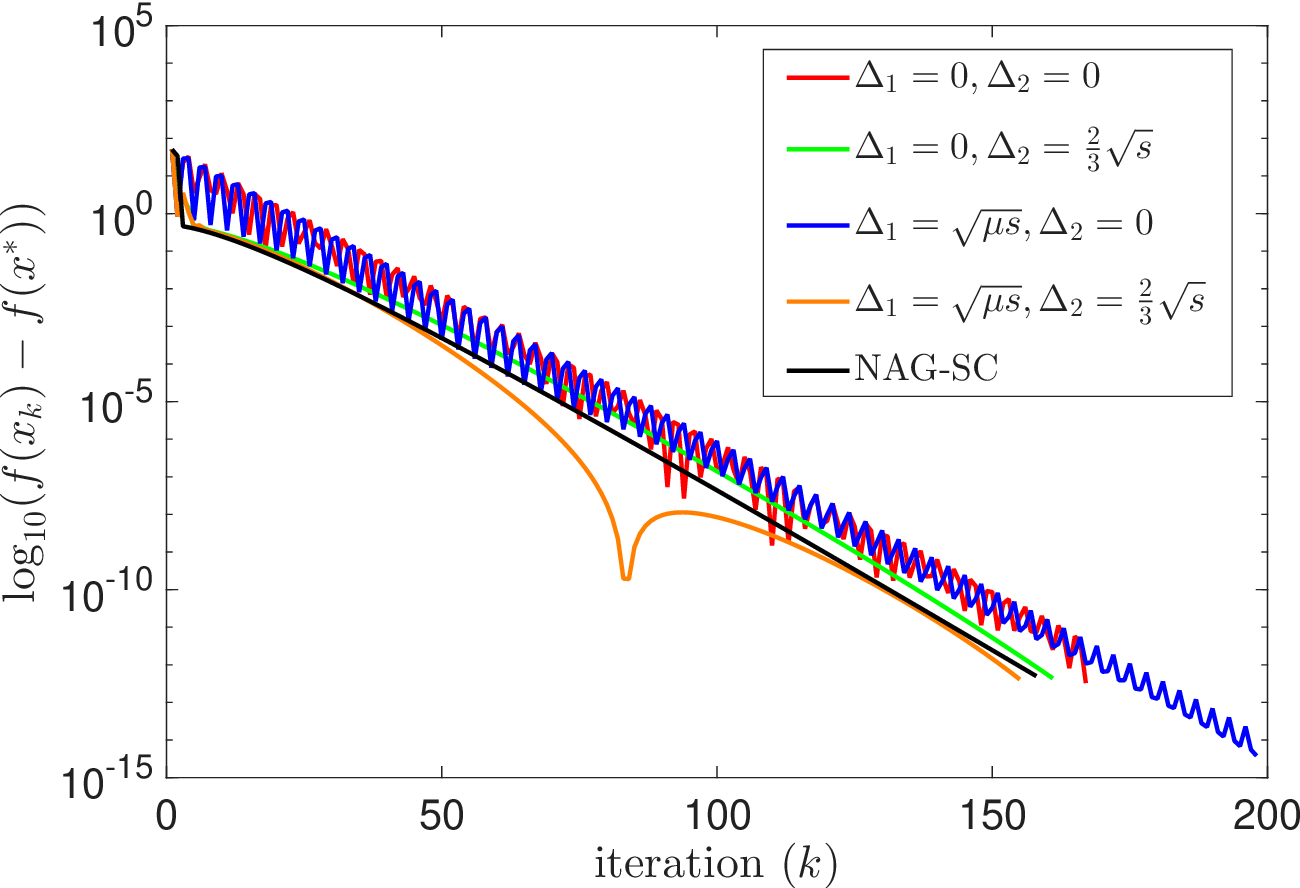}}
    \caption{Numerical comparisons of scheme \eqref{eq: iteration of direct symplectic discretization} with different $(\widehat{\Delta}_1, \widehat{\Delta}_2)$ on solving problem \eqref{eq: quadratic function for numerical experiments}.}
    \label{pic: all numerical results of xk of direct symplectic discretization}
\end{figure}

From Figure \ref{pic: all numerical results of xk of direct symplectic discretization}, it is clear that  
$\Delta_2$ is crucial in reducing oscillations.  
However, the scheme involving only the gradient-correction perturbation, i.e., $\Delta_1 =0, \Delta_2 >0$, may be slower than the scheme with properly chosen values of $\Delta_1 >0$ and $\Delta_2 >0$.  Notably, the gradient perturbation $\Delta_1\nabla f(x_{k+1})$ introduces oscillations, but it can also mitigate the negative effects of the gradient-correction perturbation $\Delta_2 (\nabla f(x_{k+1}) - \nabla f(x_k)) / \sqrt{s}$.  
Thus, an appropriate choice of both $\Delta_1 >0$ and $\Delta_2 >0$ yields a better convergence rate.  These observations are consistent with our discussions in Section \ref{subsec: direct symplectic discretization}.  

The above observations are further verified by comparing Figure  
\ref{pic: numerical results of xk of direct symplectic discretization} with Figures \ref{pic: numerical results of xk of direct symplectic discretization for additional delta1} and \ref{pic: numerical results of xk of direct symplectic discretization for additional delta2}. Specifically, increasing $\Delta_1$ from $\sqrt{\mu s}$ to $1$ intensifies oscillations but accelerates convergence.  Meanwhile, decreasing $\Delta_2$ from $\sqrt{s}$ to $2\sqrt{s}/3$ improves the convergence rate, suggesting that the gradient-correction perturbation $\Delta_2 (\nabla f(x_{k+1}) - \nabla f(x_k)) / \sqrt{s}$ can be harmful to the algorithm.  Additional tests on a convex quadratic programming problem with a larger and more general matrix $A$ are provided in Appendix \ref{sec:appdenx_numerical experiment}. The results there similarly confirm the respective roles of $\Delta_1$ and $\Delta_2$.

\subsection{Numerical experiments on $\ell_2$-regularized logistic regression}
\label{subsec: numerical experiments for direct symplectic discretization}
In this subsection, we focus on solving the following $\ell_2$-regularized logistic regression problem:
\begin{equation}
	\label{eq: logistic regression}
	\min\limits_{x\in\mathbb{R}^n} \; f(x)=\frac{1}{m} \sum_{i=1}^{m}\log(1+e^{-b_ia_i^Tx}) + \frac{\mu}{2} \|x\|_2^2,
\end{equation}
with $\{(a_i,b_i)\}_{i=1}^m$ being $m$ given feature and label pairs and $\mu >0$ being the regularization parameter. 
By simple calculations, we use the following upper bound as an estimate of the Lipschitz constant of $\nabla f$: 
\[ 
L = \frac{1}{4m}\sum_{i=1}^{m}\|a_i\|^2+\mu.
\]
It is not difficult to see that $f$ here is $\mu$-strongly convex and $L$-smooth.
Here, we set the regularization parameter $\mu = 10^{-2}$, the initial point 
\[x_0 = 0 \in \mathbb{R}^n, \quad x_1 = x_0 - \frac{s(1+\Delta_1)}{1+2\sqrt{\mu s}} \nabla f(x_0).\]
In our experiments, we solve problem \eqref{eq: logistic regression} using the pairs $\{(a_i, b_i)\}$ from the LIBSVM datasets {\bf a9a}, {\bf CINA}, and {\bf ijcnn1}. Since the exact solution to problem \eqref{eq: logistic regression} is unavailable, we use the point returned by NAG-SC under a stricter stopping criterion, $\|\nabla f(x_k)\| < 10^{-8}$, as an approximate optimal solution and denote it by \( x^* \). The detailed comparisons on the datasets {\bf CINA}, {\bf a9a} and {\bf ijcnn1} are presented in
Figures \ref{fig:cina}, \ref{fig:a9a}, and \ref{fig:ijcnn1}, respectively.

From Figure \ref{pic: numerical results of xk of direct symplectic discretization2}, Figure \ref{pic: numerical results of xk of direct symplectic discretization2 data1} and Figure \ref{pic: numerical results of xk of direct symplectic discretization2 data2}, we observe that involving only the perturbation term $\Delta_1 \nabla f(x_{k+1})$, i.e., setting $\Delta_2 =0$ in \eqref{eq: iteration of direct symplectic discretization}, results in persistent oscillations and slow convergence. In contrast, introducing a non-zero gradient correction perturbation term, i.e., setting $\Delta_2 >0$, significantly reduces oscillations and accelerates convergence. 
These observations are consistent with the theoretical results in Theorem \ref{th:symplectic 2} and Corollary \ref{symplectic 2 sufficient condition}.
The desired results can be further verified 
through the comparison between Figure \ref{pic: numerical results of xk of direct symplectic discretization2}
and Figure \ref{pic: numerical results of xk of direct symplectic discretization2 for additional delta1}, Figure \ref{pic: numerical results of xk of direct symplectic discretization2 for additional delta2} on dataset {\bf CINA}, Figure \ref{pic: numerical results of xk of direct symplectic discretization2 data1} and Figure \ref{pic: numerical results of xk of direct symplectic discretization2 data1 delta1}, Figure \ref{pic: numerical results of xk of direct symplectic discretization2 data1 delta2} on dataset {\bf a9a}, Figure \ref{pic: numerical results of xk of direct symplectic discretization2 data2} and Figure \ref{pic: numerical results of xk of direct symplectic discretization2 data2 delta1}, Figure \ref{pic: numerical results of xk of direct symplectic discretization2 data2 delta2} on dataset {\bf ijcnn1}. Notably, increasing $\Delta_1$ from $\sqrt{\mu s}$ to 1 exacerbates oscillations in function values, whereas decreasing $\Delta_2$ from $\sqrt{s}$ to $2\sqrt{s}/3$ further intensifies these oscillations.

\begin{figure}[h]
    \centering
    \subfigure[$\hat{\Delta}_1=\sqrt{\mu s}$, $\hat{\Delta}_2=\sqrt{s}$]{
    \label{pic: numerical results of xk of direct symplectic discretization2}
    \includegraphics[width=0.49\linewidth]{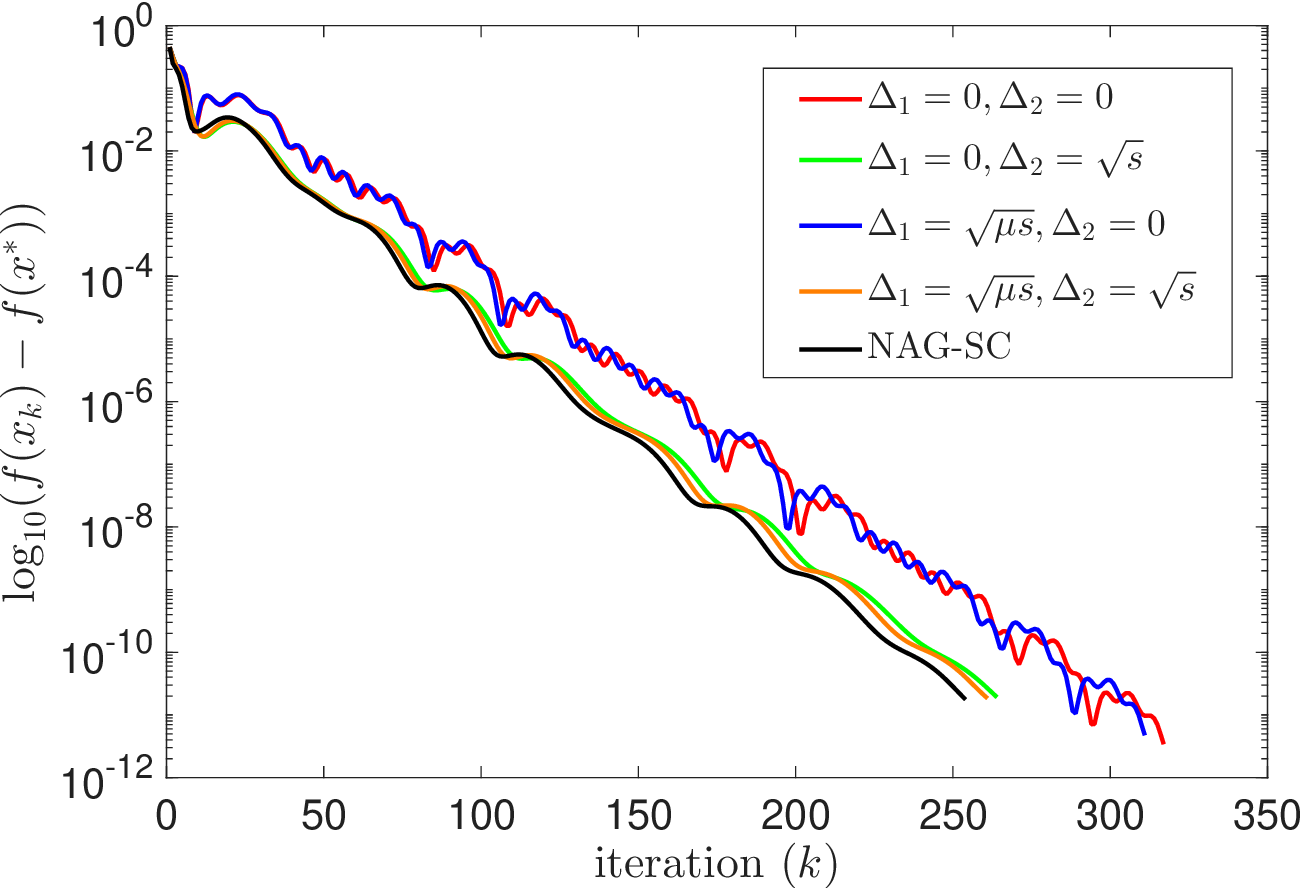}}
    \subfigure[$\hat{\Delta}_1=1$, $\hat{\Delta}_2=\sqrt{s}$]{
    \label{pic: numerical results of xk of direct symplectic discretization2 for additional delta1}
    \includegraphics[width=0.49\linewidth]{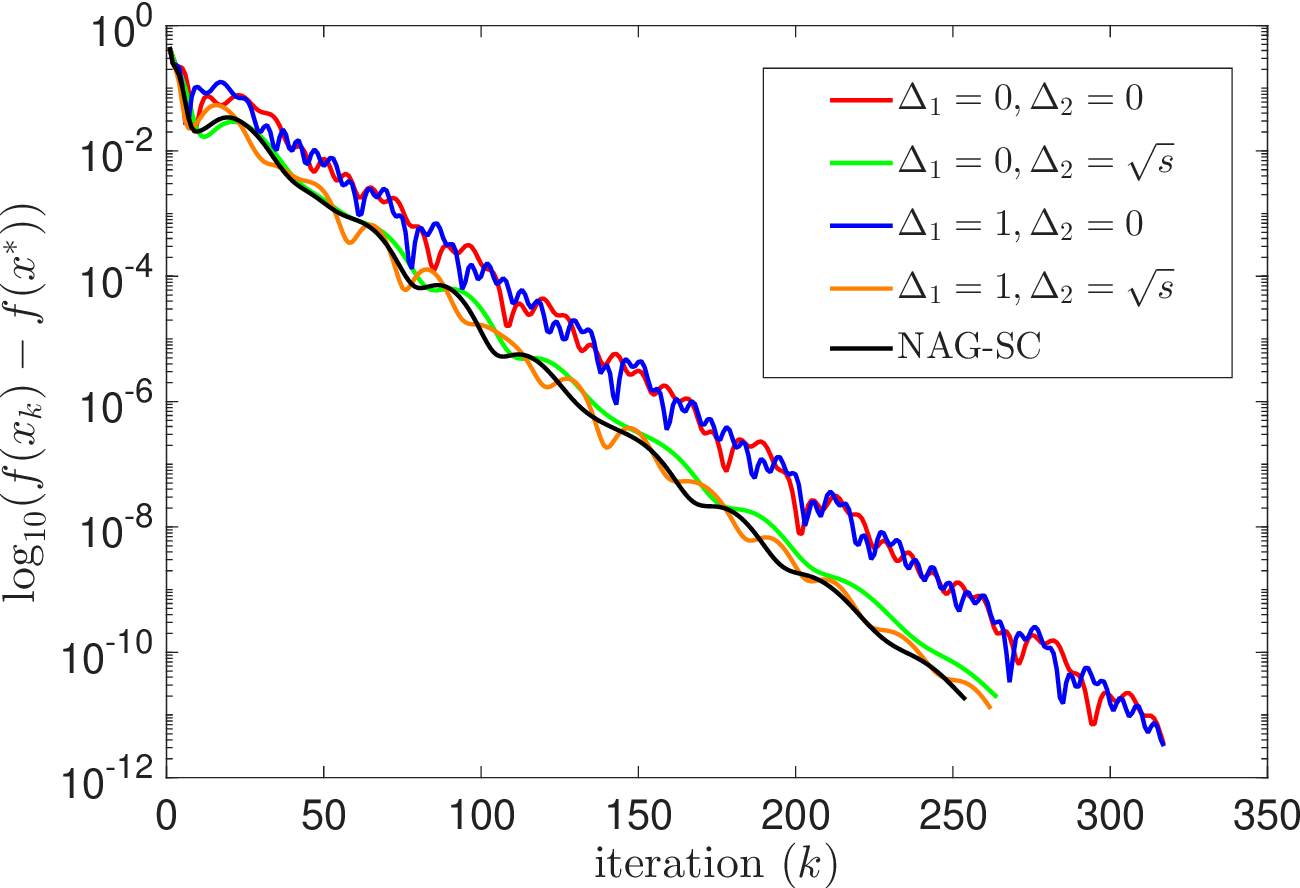}}
    \subfigure[$\hat{\Delta}_1=\sqrt{\mu s}$, $\hat{\Delta}_2=2\sqrt{s}/3$]
    {\label{pic: numerical results of xk of direct symplectic discretization2 for additional delta2}
    \includegraphics[width=0.49\linewidth]{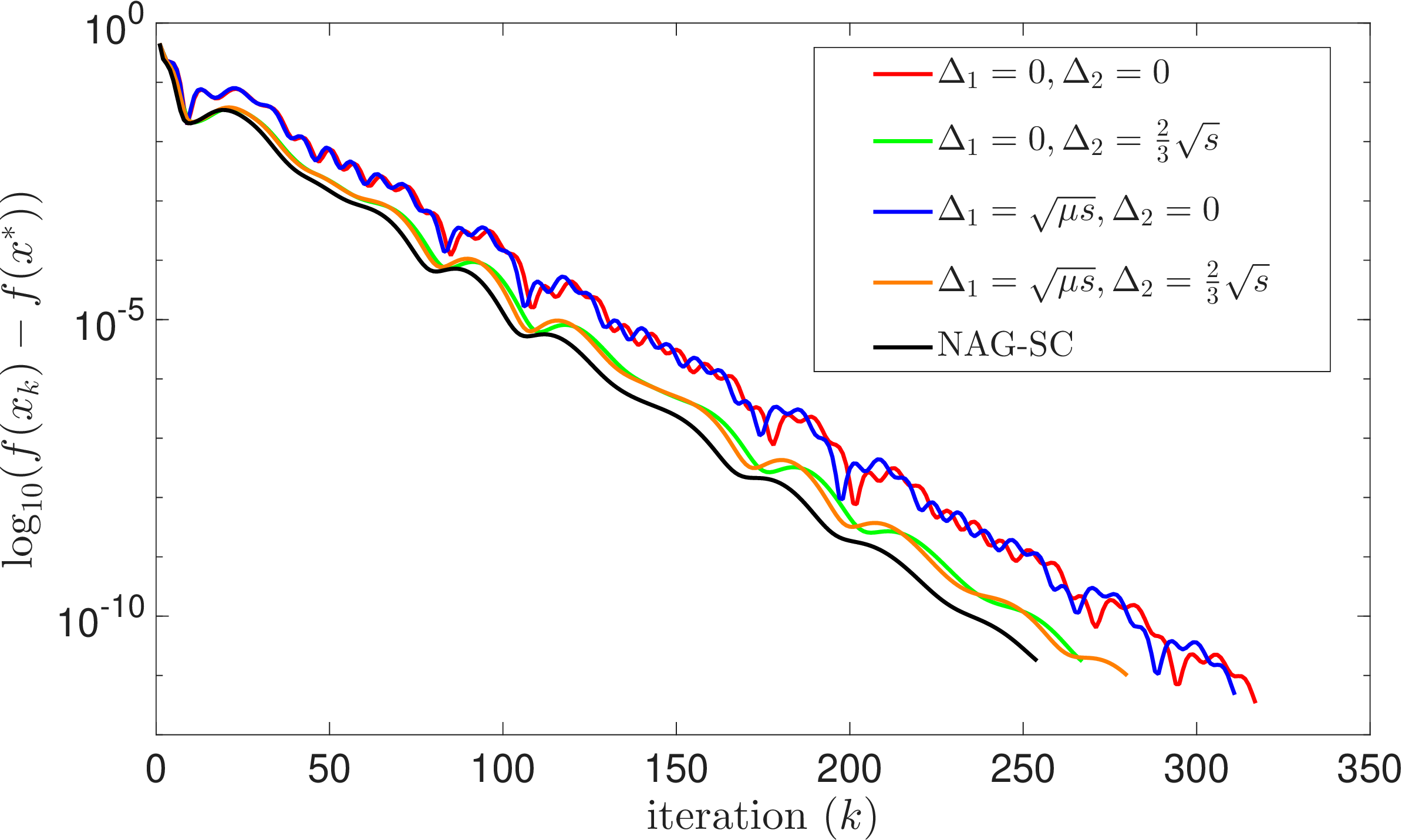}}
    \caption{Numerical comparisons of scheme \eqref{eq: iteration of direct symplectic discretization} with different $(\widehat{\Delta}_1, \widehat{\Delta}_2)$ on solving $\ell_2$-regularized logistic regression \eqref{eq: logistic regression} with dataset {\bf CINA}.}
    \label{fig:cina}
\end{figure}

\begin{figure}
    \centering
    \subfigure[$\hat{\Delta}_1=\sqrt{\mu s}$, $\hat{\Delta}_2=\sqrt{s}$]{
    \label{pic: numerical results of xk of direct symplectic discretization2 data1}
    \includegraphics[width=0.49\linewidth]{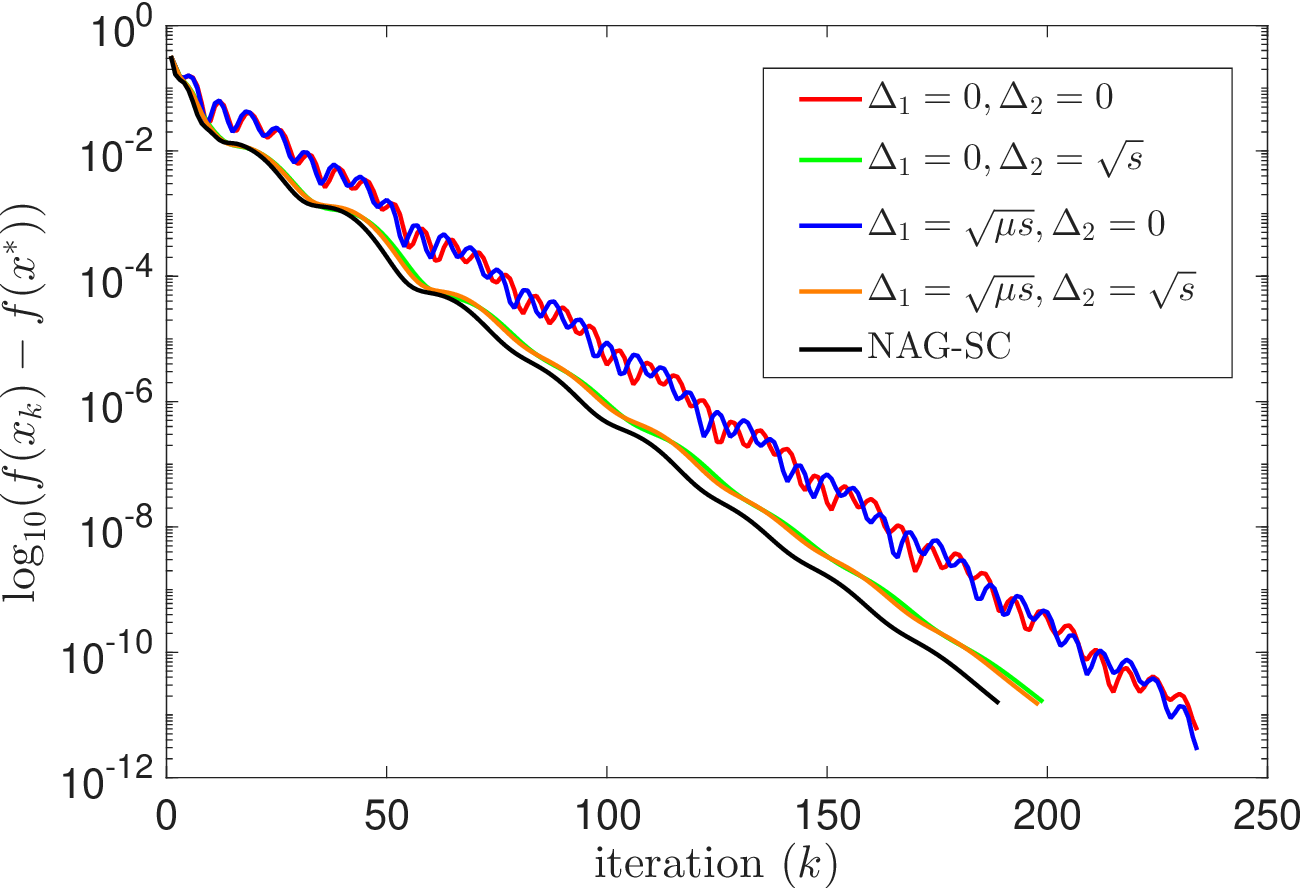}}
    \subfigure[$\hat{\Delta}_1=1$, $\hat{\Delta}_2=\sqrt{s}$]{
    \label{pic: numerical results of xk of direct symplectic discretization2 data1 delta1}
    \includegraphics[width=0.49\linewidth]{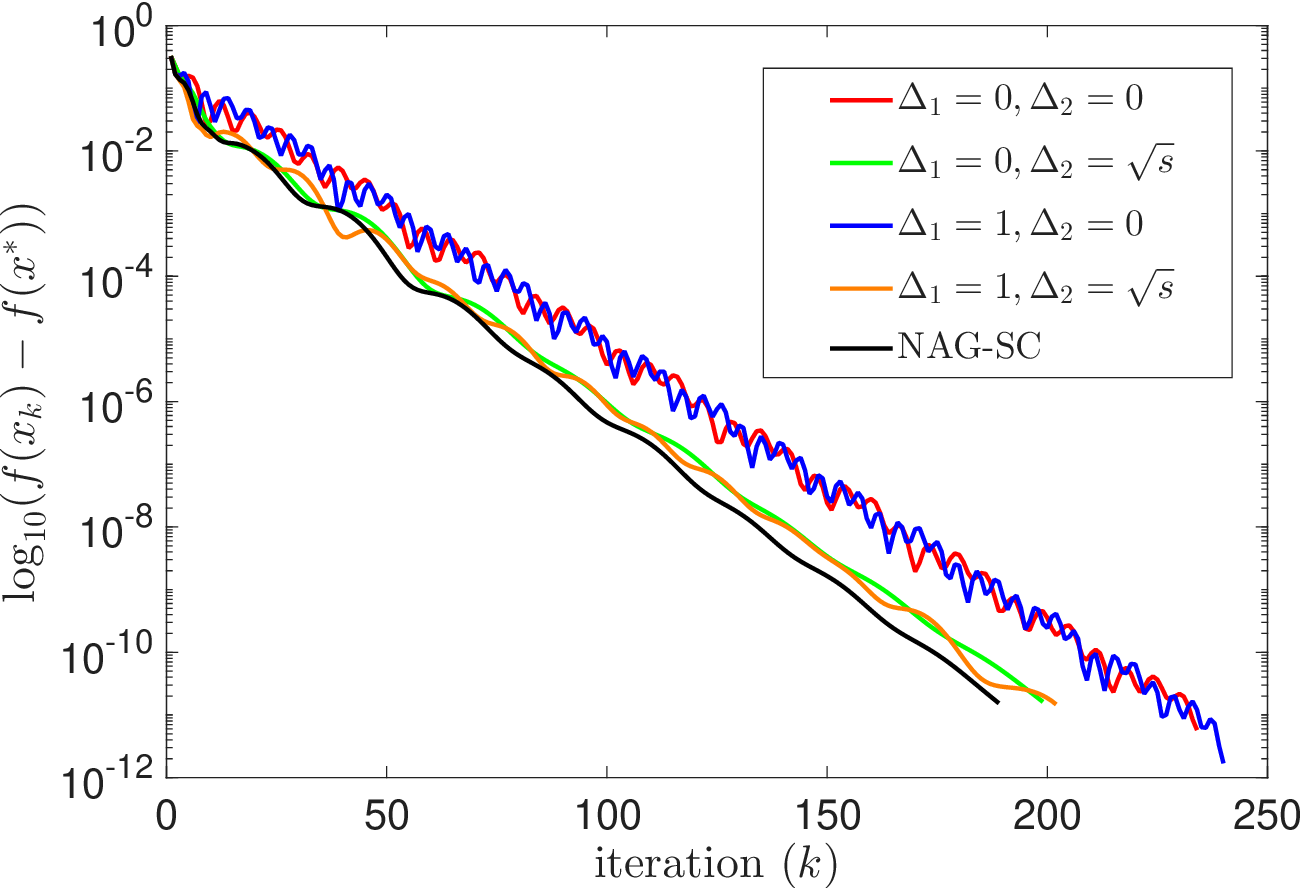}}
    \subfigure[$\hat{\Delta}_1=\sqrt{\mu s}$, $\hat{\Delta}_2=2\sqrt{s}/3$]{
    \label{pic: numerical results of xk of direct symplectic discretization2 data1 delta2}
    \includegraphics[width=0.49\linewidth]{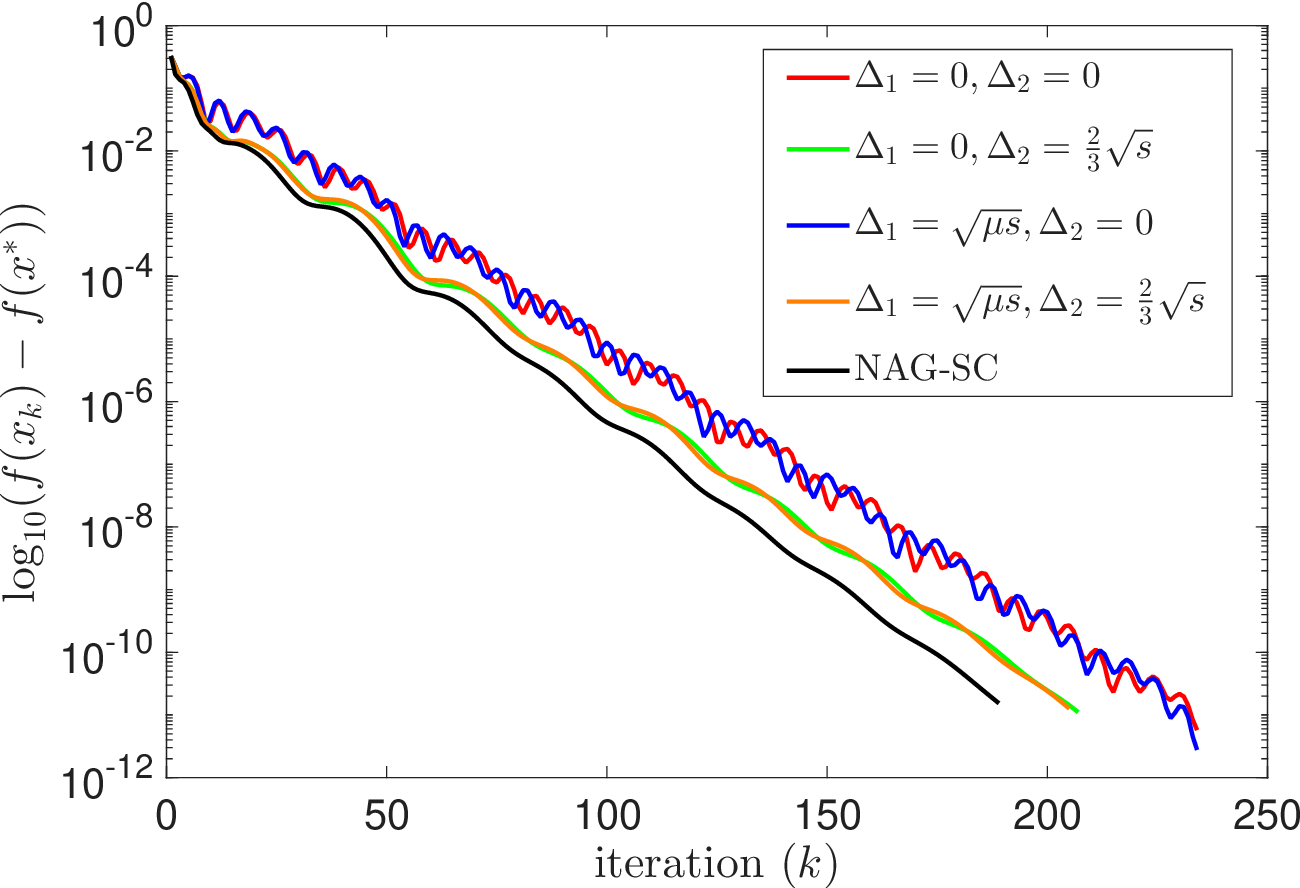}}

    \caption{Numerical comparisons of scheme \eqref{eq: iteration of direct symplectic discretization} with different $(\widehat{\Delta}_1, \widehat{\Delta}_2)$ on solving $\ell_2$-regularized logistic regression \eqref{eq: logistic regression} with dataset {\bf a9a}.}
\label{fig:a9a}
\end{figure}

\begin{figure}
    \centering
    \subfigure[$\hat{\Delta}_1=\sqrt{\mu s}$, $\hat{\Delta}_2=\sqrt{s}$]{ 
    \label{pic: numerical results of xk of direct symplectic discretization2 data2}
    \includegraphics[width=0.49\linewidth]{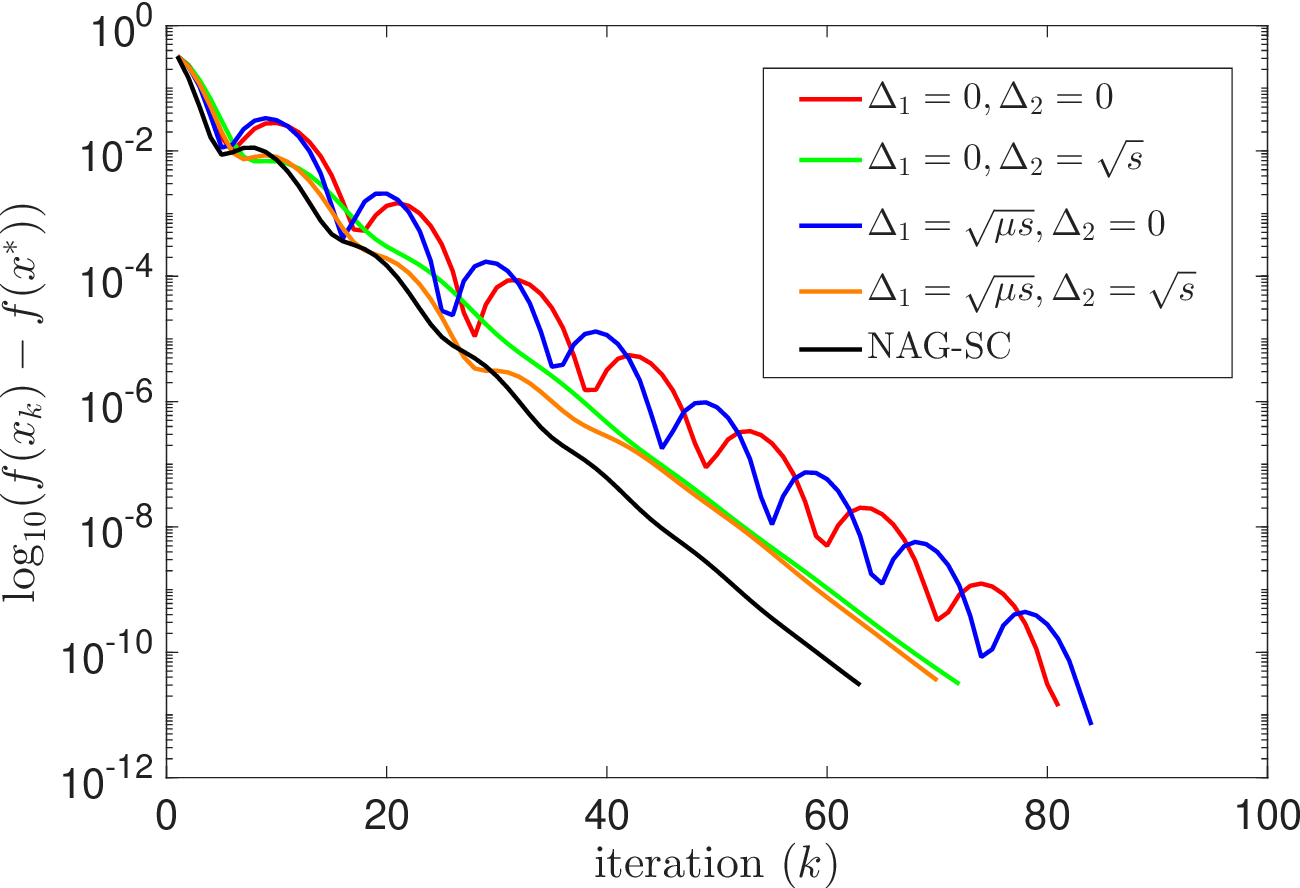}}
    \subfigure[$\hat{\Delta}_1=1$, $\hat{\Delta}_2=\sqrt{s}$]{ 
    \label{pic: numerical results of xk of direct symplectic discretization2 data2 delta1}
    \includegraphics[width=0.49\linewidth]{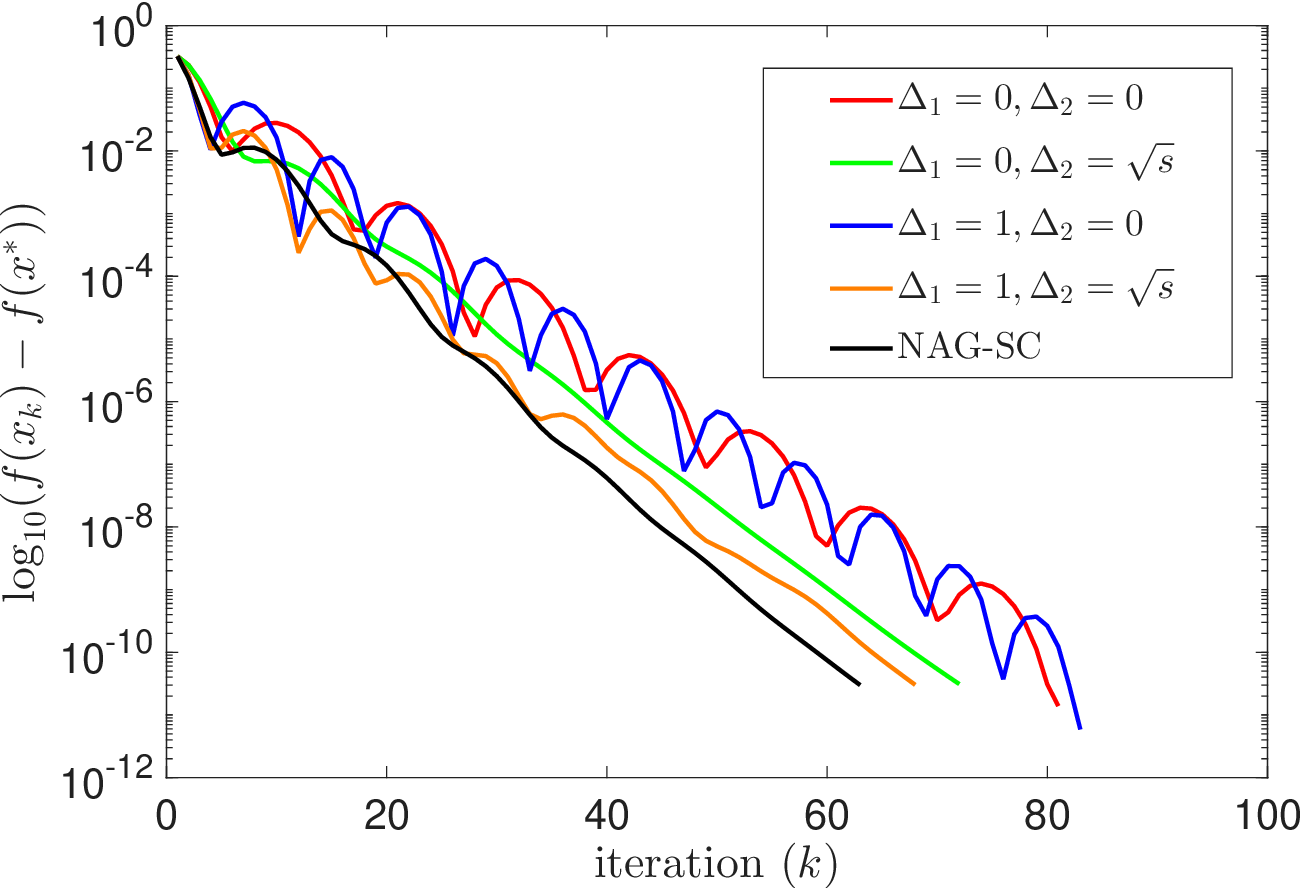}}
    \subfigure[$\hat{\Delta}_1=\sqrt{\mu s}$, $\hat{\Delta}_2=2\sqrt{s}/3$]{ 
    \label{pic: numerical results of xk of direct symplectic discretization2 data2 delta2}
    \includegraphics[width=0.49\linewidth]{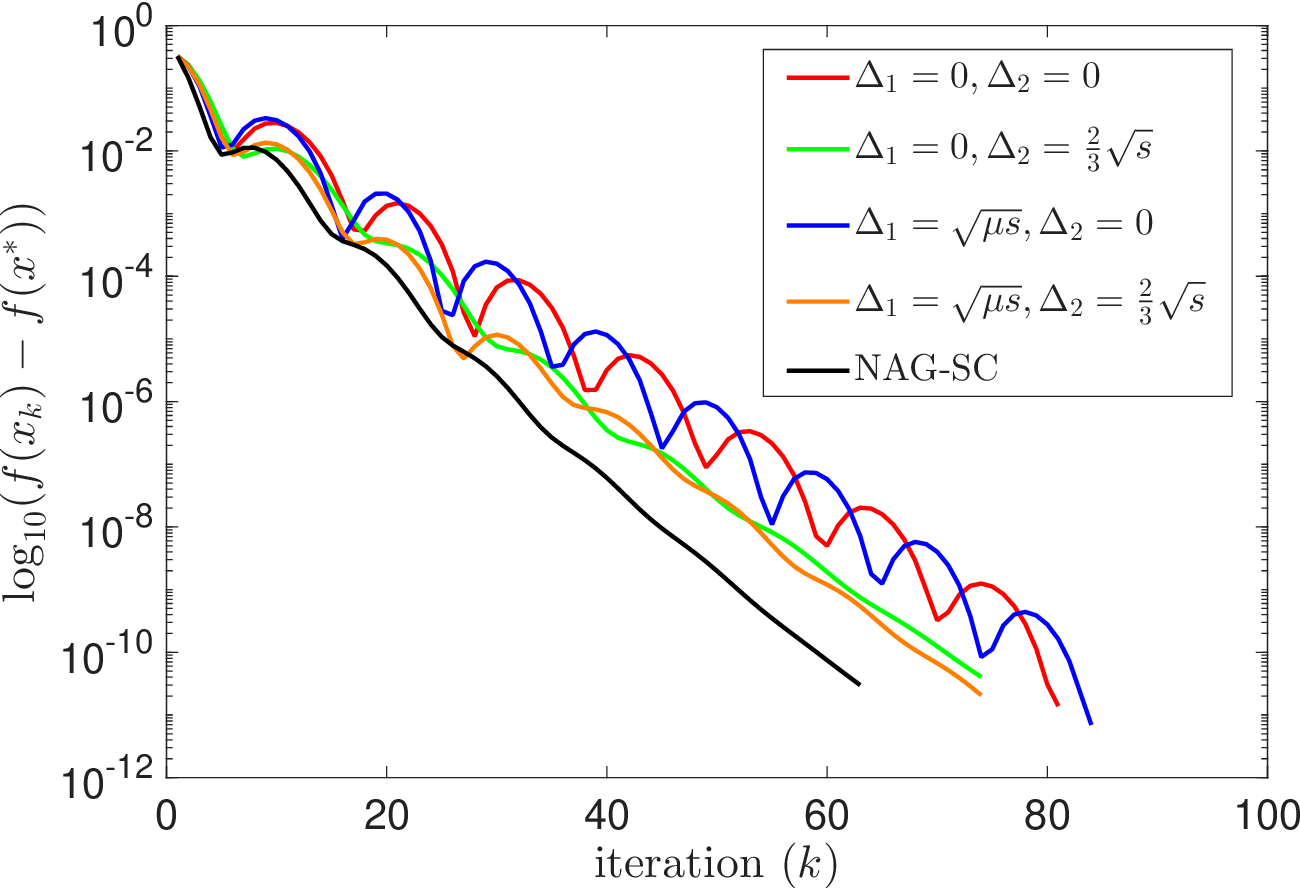}}
    \caption{Numerical comparisons of scheme \eqref{eq: iteration of direct symplectic discretization} with different $(\widehat{\Delta}_1, \widehat{\Delta}_2)$ on solving $\ell_2$-regularized logistic regression \eqref{eq: logistic regression} with dataset {\bf ijcnn1}.}
\label{fig:ijcnn1}
\end{figure}

\section{Conclusion}
\label{sec: discussion}
In this paper, to better understand the role of the gradient-correction term in accelerated algorithms, we investigate a perturbed version of the low-resolution ODE \eqref{low resolution for strong convex}. We derive appropriate choices for the perturbation parameters to ensure acceleration in both the continuous-time trajectory and its implicit and symplectic Euler discretizations. Additionally, we analyze the effects of the gradient perturbation and gradient-correction perturbation terms in detail. In particular, we show that while the gradient-correction perturbation is crucial for reducing oscillations and enabling acceleration, it may also hinder the convergence rate in certain cases. Interestingly, despite introducing oscillations, the gradient perturbation can counteract the adverse effects of the gradient-correction perturbation. As a promising direction for future work, extending our analysis to general convex composite optimization problems is highly desirable.

\section*{Data Availability}

The code and data set are available at \url{https://github.com/smq1918/perturbation-acceleration-ode}.

\section*{Declarations}

The authors declare that they have no conflicts of interest.

\bibliographystyle{siam}
\bibliography{reference}

\newpage
\appendix
\section{Analysis for the modified symplectic Euler discretization \eqref{eq: iteration of modified symplectic discretization}} \label{sec:appdenx_modified}
Similar to \eqref{Lyapunov for symplectic 2}, associated with \eqref{eq: iteration of modified symplectic discretization}, we define the following Lyapunov function $E(k)$ and obtain rate of convergence result in Theorem \ref{th:modified symplectic 2 slow}:
\begin{equation}
	\label{Lyapunov for symplectic 2 slow}
	\begin{aligned}
		E(k)=&(1+\sqrt{\mu s})^k\big(\frac{1+\Delta_1}{1-\sqrt{\mu s}}\big(f(x_k)-f(x^*)-\frac{\Delta_2\sqrt{s}}{2(1-\sqrt{\mu s})}\|\nabla f(x_k)\|^2\big)\\
		&+\frac{1}{2}\|v_k+\frac{\sqrt{\mu}}{1-\sqrt{\mu s}}(x_{k+1}-x^*)+\frac{\Delta_2}{1-\sqrt{\mu s}}\nabla f(x_k)\|^2\big).
	\end{aligned}
\end{equation}
\begin{THM}
	\label{th:modified symplectic 2 slow}
	Suppose that $f$ is $\mu$-strongly convex and $L$-smooth. If the non-negative perturbation
	parameters $\Delta_1, \Delta_2$ and step size $s>0$ satisfy the conditions:
	\begin{equation}
		\label{conditions of modified symplectic 2 slow}
		\begin{aligned}
			&\, (1)\frac{\Delta_2\sqrt{s}}{1-\sqrt{\mu s}}\leqslant \frac{1}{L};\\
			&\, (2)\frac{\sqrt{s}}{2}(1+\Delta_1)\leqslant \Delta_2\leqslant \sqrt{s}(1+\Delta_1);\\
			&\, (3) 1+\Delta_1 \geqslant \frac{1}{1-\sqrt{\mu s}},
		\end{aligned}
	\end{equation}
	then for any initial points $x_0$, $v_0$, the sequence $\{x^k\}$ generated by \eqref{eq: iteration of modified symplectic discretization} satisfies that
	\begin{equation*}
		f(x_k)-f(x^*)-\frac{\Delta_2\sqrt{s}}{2(1-\sqrt{\mu s})}\|\nabla f(x_k)\|^2 \leqslant \frac{1-\sqrt{\mu s}}{1+\Delta_1}(1+\sqrt{\mu s})^{-k}E(0).
	\end{equation*}
	Thus, if in addition $\Delta_2\sqrt{s}<(1-\sqrt{\mu s})/L$, then
	\begin{equation}
		f(x_k)-f(x^*) \leqslant \Big(1-\frac{\Delta_2\sqrt{s}L}{1-\sqrt{\mu s}}\Big)^{-1}\frac{1-\sqrt{\mu s}}{1+\Delta_1}(1+\sqrt{\mu s})^{-k}E(0).
	\end{equation}
\end{THM}
\begin{proof}
    {The proof here is similar to the one for Theorem \ref{th:symplectic 2}. We start by showing that $E(k)$ is nonincreasing. 
        \begin{equation}
        \label{eq: diffEk in modified symplectic discretization}
        \begin{aligned}
        &(1+\sqrt{\mu s})^{-k}\big(E(k+1)-E(k)\big)\\
        =&\frac{1+\Delta_1}{1-\sqrt{\mu s}}\big(f(x_{k+1})-f(x_k)\big)+\sqrt{\mu s}\frac{1+\Delta_1}{1-\sqrt{\mu s}}\big(f(x_{k+1})-f(x^*)\big)\\
		&-(1+\sqrt{\mu s})\frac{(1+\Delta_1)\Delta_2\sqrt{s}}{2(1-\sqrt{\mu s})^2}\|\nabla f(x_{k+1})\|^2+\frac{(1+\Delta_1)\Delta_2\sqrt{s}}{2(1-\sqrt{\mu s})^2}\|\nabla f(x_k)\|^2\\
        &+M_1^k+M_2^k,
        \end{aligned}
        \end{equation}
        where 
        \begin{equation*}
        \begin{aligned}
            M_1^k={}&\frac{1}{2}\|v_{k+1}+\frac{\sqrt{\mu}}{1-\sqrt{\mu s}}(x_{k+2} -x^*)+\frac{\Delta_2}{1-\sqrt{\mu s}}\nabla f(x_{k+1})\|^2 \\
            {}&-\frac{1}{2}\|v_k+\frac{\sqrt{\mu}}{1-\sqrt{\mu s}}(x_{k+1}-x^*)+\frac{\Delta_2}{1-\sqrt{\mu s}}\nabla f(x_k)\|^2,
        \end{aligned}
        \end{equation*}
        and 
        \begin{equation*}
            M_2^k=\frac{\sqrt{\mu s}}{2}\|v_{k+1}+\frac{\sqrt{\mu}}{1-\sqrt{\mu s}}(x_{k+2}-x^*)+\frac{\Delta_2}{1-\sqrt{\mu s}}\nabla f(x_{k+1})\|^2.
        \end{equation*}
        Similar to the proof for Theorem \ref{th:symplectic 2}, 
        by utilizing \eqref{strong disturb symplectic 2 iteration} and $x_{k+2}=x_{k+1}+\sqrt{s}v_{k+1}$, we can get
        \begin{align*}
        M_1^k+M_2^k = 
        &-\frac{\sqrt{\mu s}(1+\sqrt{\mu s})}{2(1-\sqrt{\mu s})^2}\|v_{k+1}\|^2
        -\frac{(1+\Delta_1)\sqrt{s}}{(1-\sqrt{\mu s})^2}(1+\sqrt{\mu s})\langle v_{k+1},\nabla f(x_{k+1})\rangle\\
        &-\Big\{\frac{\sqrt{\mu s}}{1-\sqrt{\mu s}}(1+\Delta_1)+\frac{\mu\sqrt{s}}{(1-\sqrt{\mu s})^2}\big[(1+\Delta_1)\sqrt{s}-\Delta_2\big]\Big\}\langle x_{k+1}-x^*,\nabla f(x_{k+1})\rangle\\
        &+\big[-\frac{(1+\Delta_1)\Delta_2\sqrt{s}}{(1-\sqrt{\mu s})^2}-\frac{(1+\Delta_1)^2}{2(1-\sqrt{\mu s})^2}s+\frac{\sqrt{\mu s}\Delta_2^2}{2(1-\sqrt{\mu s})^2}\big]\|\nabla f(x_{k+1})\|^2\\
        &+\frac{\sqrt{\mu s}}{2}\frac{\mu}{(1-\sqrt{\mu s})^2}\|x_{k+1}-x^*\|^2.
        \end{align*}
        Since
        \begin{equation*}
        \left\{
        \begin{aligned}
        & (1+\sqrt{\mu s})v_{k+1}=(1-\sqrt{\mu s})v_k-(1+\Delta_1)\sqrt{s}\nabla f(x_{k+1})-\Delta_2\big(\nabla f(x_{k+1})-\nabla f(x_k)\big),\\
        & \sqrt{s}v_k=x_{k+1}-x_k,
        \end{aligned}
        \right. 
        \end{equation*}
        and using
        \begin{equation*}
	\langle \nabla     
        f(x_{k+1}),\nabla f(x_{k+1})-\nabla f(x_k)\rangle=\frac{1}{2}\|\nabla f(x_{k+1})-\nabla f(x_k)\|^2+\frac{1}{2}\|\nabla f(x_{k+1})\|^2-\frac{1}{2}\|\nabla f(x_k)\|^2,
	\end{equation*}
        we have
        \begin{align*}
            &\frac{(1+\Delta_1)\sqrt{s}}{(1-\sqrt{\mu s})^2}(1+\sqrt{\mu s})\langle v_{k+1}, \nabla f(x_{k+1}) \rangle\\
            =&\frac{1+\Delta_1}{1-\sqrt{\mu s}}\langle x_{k+1}-x_k,\nabla f(x_{k+1}) \rangle - \frac{(1+\Delta_1)^2s}{(1-\sqrt{\mu s})^2}\|\nabla f(x_{k+1})\|^2\\
            &-\frac{(1+\Delta_1)\Delta_2\sqrt{s}}{2(1-\sqrt{\mu s})^2}\big(\|\nabla f(x_{k+1})-\nabla f(x_k)\|^2+\|\nabla f(x_{k+1})\|^2-\|\nabla f(x_k)\|^2\big).
        \end{align*}
        Then, it holds that  
        \begin{equation}
        \label{eq: M1M2 in modified symplectic discretization}
        \begin{aligned}
                M_1^k+M_2^k=&-\frac{\sqrt{\mu s}(1+\sqrt{\mu s})}{2(1-\sqrt{\mu s})^2}\|v_{k+1}\|^2
        -\frac{1+\Delta_1}{1-\sqrt{\mu s}}\langle x_{k+1}-x_k,\nabla f(x_{k+1})\rangle\\
        &+\frac{(1+\Delta_1)\Delta_2\sqrt{s}}{2(1-\sqrt{\mu s})^2}\big(\|\nabla f(x_{k+1})-\nabla f(x_k)\|^2+\|\nabla f(x_{k+1})\|^2-\|\nabla f(x_k)\|^2\big)\\
        &-\Big\{\frac{\sqrt{\mu s}}{1-\sqrt{\mu s}}(1+\Delta_1)+\frac{\mu\sqrt{s}}{(1-\sqrt{\mu s})^2}\big[(1+\Delta_1)\sqrt{s}-\Delta_2\big]\Big\}\langle x_{k+1}-x^*,\nabla f(x_{k+1})\rangle\\
        &+\big[-\frac{(1+\Delta_1)\Delta_2\sqrt{s}}{(1-\sqrt{\mu s})^2}+\frac{(1+\Delta_1)^2}{2(1-\sqrt{\mu s})^2}s+\frac{\sqrt{\mu s}\Delta_2^2}{2(1-\sqrt{\mu s})^2}\big]\|\nabla f(x_{k+1})\|^2\\
        &+\frac{\sqrt{\mu s}}{2}\frac{\mu}{(1-\sqrt{\mu s})^2}\|x_{k+1}-x^*\|^2.
        \end{aligned}
       \end{equation}
       By substituting \eqref{eq: M1M2 in modified symplectic discretization} into \eqref{eq: diffEk in modified symplectic discretization}, we see that
       \begin{align*}
           &(1+\sqrt{\mu s})^{-k}\big(E(k+1)-E(k)\big)\\
           =&\frac{1+\Delta_1}{1-\sqrt{\mu s}}\big(f(x_{k+1})-f(x_k)\big)+\sqrt{\mu s}\frac{1+\Delta_1}{1-\sqrt{\mu s}}\big(f(x_{k+1})-f(x^*)\big)\\
		&-(1+\sqrt{\mu s})\frac{(1+\Delta_1)\Delta_2\sqrt{s}}{2(1-\sqrt{\mu s})^2}\|\nabla f(x_{k+1})\|^2+\frac{(1+\Delta_1)\Delta_2\sqrt{s}}{2(1-\sqrt{\mu s})^2}\|\nabla f(x_k)\|^2\\
        &-\frac{\sqrt{\mu s}(1+\sqrt{\mu s})}{2(1-\sqrt{\mu s})^2}\|v_{k+1}\|^2
        -\frac{1+\Delta_1}{1-\sqrt{\mu s}}\langle x_{k+1}-x_k,\nabla f(x_{k+1})\rangle\\
        &+\frac{(1+\Delta_1)\Delta_2\sqrt{s}}{2(1-\sqrt{\mu s})^2}\big(\|\nabla f(x_{k+1})-\nabla f(x_k)\|^2+\|\nabla f(x_{k+1})\|^2-\|\nabla f(x_k)\|^2\big)\\
        &-\Big\{\frac{\sqrt{\mu s}}{1-\sqrt{\mu s}}(1+\Delta_1)+\frac{\mu\sqrt{s}}{(1-\sqrt{\mu s})^2}\big[(1+\Delta_1)\sqrt{s}-\Delta_2\big]\Big\}\langle x_{k+1}-x^*,\nabla f(x_{k+1})\rangle\\
        &+\big[-\frac{(1+\Delta_1)\Delta_2\sqrt{s}}{(1-\sqrt{\mu s})^2}+\frac{(1+\Delta_1)^2}{2(1-\sqrt{\mu s})^2}s+\frac{\sqrt{\mu s}\Delta_2^2}{2(1-\sqrt{\mu s})^2}\big]\|\nabla f(x_{k+1})\|^2\\
        &+\frac{\sqrt{\mu s}}{2}\frac{\mu}{(1-\sqrt{\mu s})^2}\|x_{k+1}-x^*\|^2\\
        =&\frac{1+\Delta_1}{1-\sqrt{\mu s}}\big(f(x_{k+1})-f(x_k)+\langle x_k-x_{k+1}, \nabla f(x_{k+1})\rangle + \frac{\Delta_2\sqrt{s}}{2(1-\sqrt{\mu s})}\|\nabla f(x_{k+1})-\nabla f(x_k)\|^2\big)\\
        &+\sqrt{\mu s}\frac{1+\Delta_1}{1-\sqrt{\mu s}}\big(f(x_{k+1})-f(x^*)+\langle x^*-x_{k+1}, \nabla f(x_{k+1})\big)+\frac{\mu}{2}\frac{\sqrt{\mu s}}{(1-\sqrt{\mu s})^2}\|x_{k+1}-x^*\|^2\\
        &-\frac{\sqrt{\mu s}(1+\sqrt{\mu s})}{2(1-\sqrt{\mu s})^2}\|v_{k+1}\|^2-\frac{\mu\sqrt{s}}{(1-\sqrt{\mu s})^2}\big[(1+\Delta_1)\sqrt{s}-\Delta_2\big]\langle x_{k+1}-x^*, \nabla f(x_{k+1})\rangle\\
        &+\big(-\frac{(1+\Delta_1)\Delta_2\sqrt{s}}{2(1-\sqrt{\mu s})^2}(2+\sqrt{\mu s})+\frac{(1+\Delta_1)^2}{2(1-\sqrt{\mu s})^2}s+\frac{\sqrt{\mu s}\Delta_2^2}{2(1-\sqrt{\mu s})^2}\big)\|\nabla f(x_{k+1})\|^2.
       \end{align*}
    }
    From the $\mu$-strong convexity and $L$-smoothness of $f$, we know 
	\begin{equation*}
		\left\{
		\begin{aligned}
			&f(x_{k+1})-f(x_k)+\langle \nabla f(x_{k+1}),x_k-x_{k+1}\rangle \leqslant -\frac{1}{2L}\|\nabla f(x_{k+1})-\nabla f(x_k)\|^2,\\
			&f(x_{k+1})-f(x^*)+\langle \nabla f(x_{k+1}),x^*-x_{k+1}\rangle\leqslant -\frac{\mu}{2}\|x_{k+1}-x^*\|^2,
		\end{aligned}
		\right.
	\end{equation*}
which further yields that
{
        \begin{align*}
            &(1+\sqrt{\mu s})^{-k}\big(E(k+1)-E(k)\big)\\
            \leqslant & \frac{1+\Delta_1}{2(1-\sqrt{\mu s})}\big(-\frac{1}{L}+\frac{\Delta_2\sqrt{s}}{1-\sqrt{\mu s}}\big)\|\nabla f(x_{k+1})-\nabla f(x_{k})\|^2\\
            &+\frac{\mu\sqrt{\mu s}}{2(1-\sqrt{\mu s})}\big(-(1+\Delta_1)+\frac{1}{1-\sqrt{\mu s}}\big)\|x_{k+1}-x^*\|^2\\
            &-\frac{\sqrt{\mu s}(1+\sqrt{\mu s})}{2(1-\sqrt{\mu s})^2}\|v_{k+1}\|^2-\frac{\mu\sqrt{s}}{(1-\sqrt{\mu s})^2}\big((1+\Delta_1)\sqrt{s}-\Delta_2\big)\langle x_{k+1}-x^*,\nabla f(x_{k+1})\rangle\\
			&+\big(-\frac{(1+\Delta_1)\Delta_2\sqrt{s}}{2(1-\sqrt{\mu s})^2}(2+\sqrt{\mu s})+\frac{(1+\Delta_1)^2}{2(1-\sqrt{\mu s})^2}s+\frac{\sqrt{\mu s}\Delta_2^2}{2(1-\sqrt{\mu s})^2}\big)\|\nabla f(x_{k+1})\|^2.
        \end{align*}
    }
Note that the coefficient before the last term $\|\nabla f(x_{k+1})\|^2$ can be rewritten as follows:
	\begin{align*}
		&-\frac{(1+\Delta_1)\Delta_2\sqrt{s}}{2(1-\sqrt{\mu s})^2}(2+\sqrt{\mu s})+\frac{(1+\Delta_1)^2}{2(1-\sqrt{\mu s})^2}s+\frac{\sqrt{\mu s}\Delta_2^2}{2(1-\sqrt{\mu s})^2}\\
		=&-\frac{(1+\Delta_1)\sqrt{s}}{(1-\sqrt{\mu s})^2}\big(\Delta_2-\frac{1}{2}\sqrt{s}(1+\Delta_1)\big)-\frac{\sqrt{\mu s}\Delta_2}{2(1-\sqrt{\mu s})^2}\big(\sqrt{s}(1+\Delta_1)-\Delta_2\big).
	\end{align*}
Then, from \ref{conditions of modified symplectic 2 slow}, we have 
	\begin{equation*}
		E(k+1)\leqslant E(k), \quad \forall\, k\ge 0,
	\end{equation*}   
    which, together with \eqref{Lyapunov for symplectic 2 slow}, implies that
	\begin{align*}
		f(x_k)-f(x^*)-\frac{\Delta_2\sqrt{s}}{2(1-\sqrt{\mu s})}\|\nabla f(x_k)\|^2 \leqslant \frac{1-\sqrt{\mu s}}{1+\Delta_1}(1+\sqrt{\mu s})^{-k}E(k)\leqslant \frac{1-\sqrt{\mu s}}{1+\Delta_1}(1+\sqrt{\mu s})^{-k}E(0).
	\end{align*}{
    Then, the $L$-smoothness of $f$ and the optimality of $x^*$ further yield that
     the following inequality holds if $\Delta_2\sqrt{s}L/(1-\sqrt{\mu s})<1$:
    \begin{align*}
        f(x_k)-f(x^*) \leqslant &\Big(1-\frac{\Delta_2\sqrt{s}L}{1-\sqrt{\mu s}}\Big)^{-1} \big( f(x_k)-f(x^*)-\frac{\Delta_2\sqrt{s}}{2(1-\sqrt{\mu s})}\|\nabla f(x_k)\|^2\big)\\
        \leqslant & \Big(1-\frac{\Delta_2\sqrt{s}L}{1-\sqrt{\mu s}}\Big)^{-1}\frac{1-\sqrt{\mu s}}{1+\Delta_1}(1+\sqrt{\mu s})^{-k}E(0).
    \end{align*}
    }
	This completes the proof of the theorem.
\end{proof}

Again, the roles of $\Delta_1$ and $\Delta_2$ in achieving the acceleration can be discussed. We omit the detailed discussions here, but note that the roles of $\Delta_1$ and $\Delta_2$ in this context are similar to those in the previously discussed scheme resulted from the symplectic discretization.
Similar to Algorithm \ref{algorithm:symplectic discretization 1}, we can set $s = 1/(4L)$, $\Delta_1 = {\sqrt{\mu}/(2\sqrt{L}-\sqrt{\mu})}$, choose $\Delta_2$ satisfying:
\begin{equation*}
	\frac{1}{2(2\sqrt{L}-\sqrt{\mu})}\leqslant \Delta_2 \leqslant \frac{1}{2\sqrt{L}-\sqrt{\mu}},
\end{equation*}
and obtaining a class of accelerated algorithms with the following convergence rate:
\[
f(x_k)-f(x^*)={\cal O}((1+1/2\sqrt{\mu/L})^{-k}).
\] 
The algorithm is summarized in Algorithm \ref{algorithm:symplectic discretization 2}.

\begin{algorithm}
	\caption{Optimization algorithm derived from the modified symplectic Euler discretization}
	\label{algorithm:symplectic discretization 2}
	\begin{algorithmic}[1]
		\State{Choose the initial point $x_0 \in \cH$,  $\Delta_1=\sqrt{\mu}/(2\sqrt{L}-\sqrt{\mu})$, and $\Delta_2$ satisfying $1/(4\sqrt{L}-2\sqrt{\mu})\leqslant \Delta_2\leqslant 1/(2\sqrt{L}-\sqrt{\mu})$, and set $x_1=x_0$.}
		\For{$k=1,2,\ldots$}
		\State{$x_{k+1}=x_k+\frac{2\sqrt{L}-\sqrt{\mu}}{2\sqrt{L}+\sqrt{\mu}}(x_k-x_{k-1})-\frac{1}{4L-\mu}\nabla f(x_{k})-\frac{\Delta_2}{2\sqrt{L}-\sqrt{\mu}}\big(\nabla f(x_{k})-\nabla f(x_{k-1})\big)$.}
		\EndFor
	\end{algorithmic}
\end{algorithm}

\section{Additional numerical experiments on quadratic programming}
\label{sec:appdenx_numerical experiment}
Below, we test our algorithms for solving the quadratic programming problem \eqref{eq: quadratic function for numerical experiments} with 
$A \in \mathbb{R}^{100\times 100}$ given by 
    \begin{equation}
        \label{eq: case matrix}
        A=Q {\rm Diag}\left(\left[\mu; \mu\left(\frac{L}{\mu}\right)^{\frac{1}{n-1}}; \ldots; \mu\left(\frac{L}{\mu}\right)^{\frac{n-2}{n-1}}; L\right]\right) Q^T,
    \end{equation}
    where $\mu=1$, $L=100$, $n=100$, and $Q$ is a randomly generated orthogonal matrix. Specifically, an $n \times n$ matrix is first generated with independent and identically distributed entries drawn from the standard normal distribution. The orthogonal matrix $Q$ is then obtained by performing a QR decomposition on this random matrix. In the tests, the initial point is chosen as
    \begin{equation*}
        x_0 = Q  [1;1;\cdots;1;1], \quad x_1 = x_0 - \frac{s(1+\Delta_1)}{1+2\sqrt{\mu s}} \nabla f(x_0).
    \end{equation*}
    The detailed results are presented in Figure \ref{pic:larger matrix}.  
    Compared with the results obtained from Figure \ref{pic: all numerical results of xk of direct symplectic discretization}, a similar conclusion can be drawn regarding the roles of  $\Delta_1$ and $\Delta_2$.
    
    \begin{figure}[htbp]
    \centering
    \subfigure[$\hat{\Delta}_1=\sqrt{\mu s}$, $\hat{\Delta}_2=\sqrt{s}$]{ 
    \centering
    \includegraphics[width=0.49\linewidth]{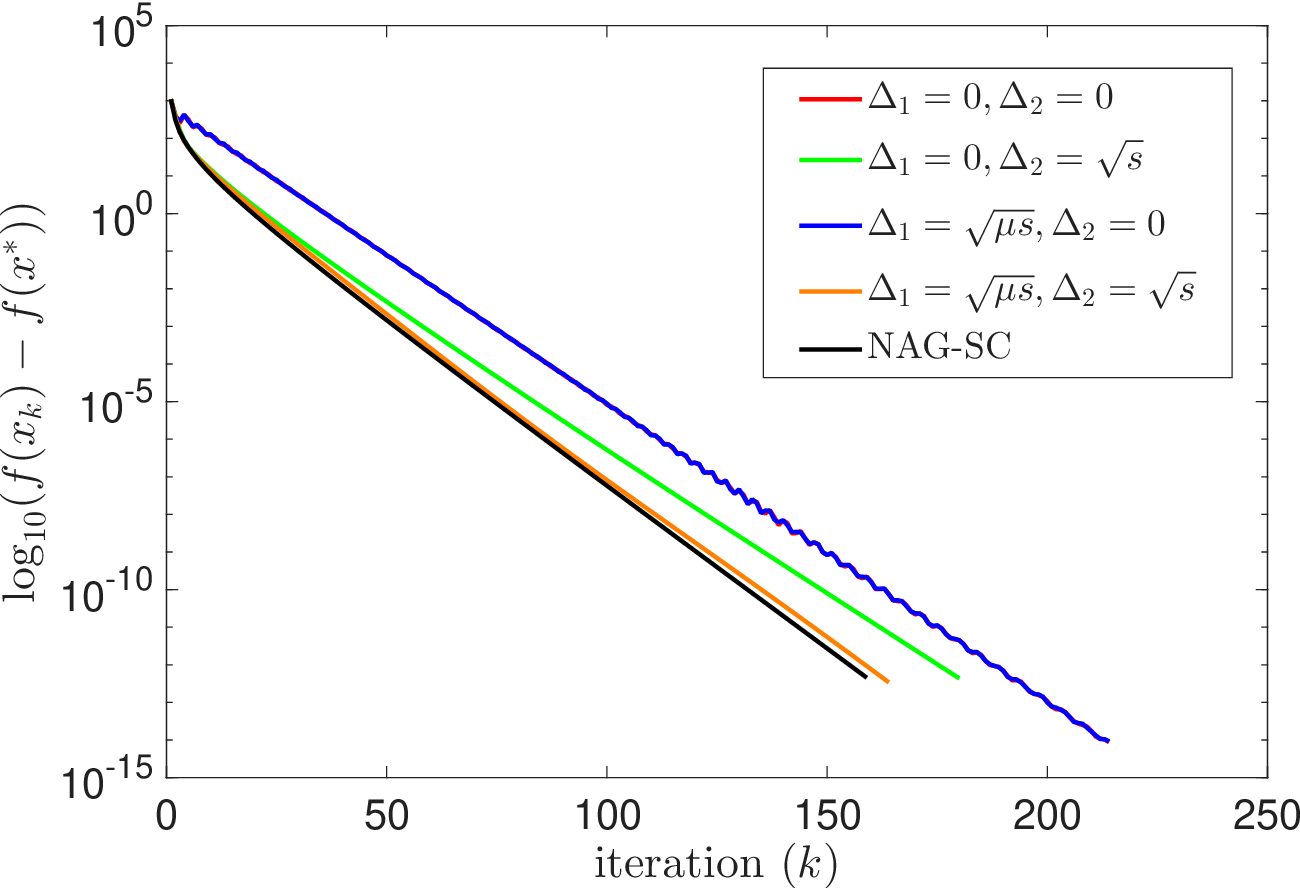}}
    \subfigure[$\hat{\Delta}_1=1$, $\hat{\Delta}_2=\sqrt{s}$]{
    \centering
    \includegraphics[width=0.49\linewidth]{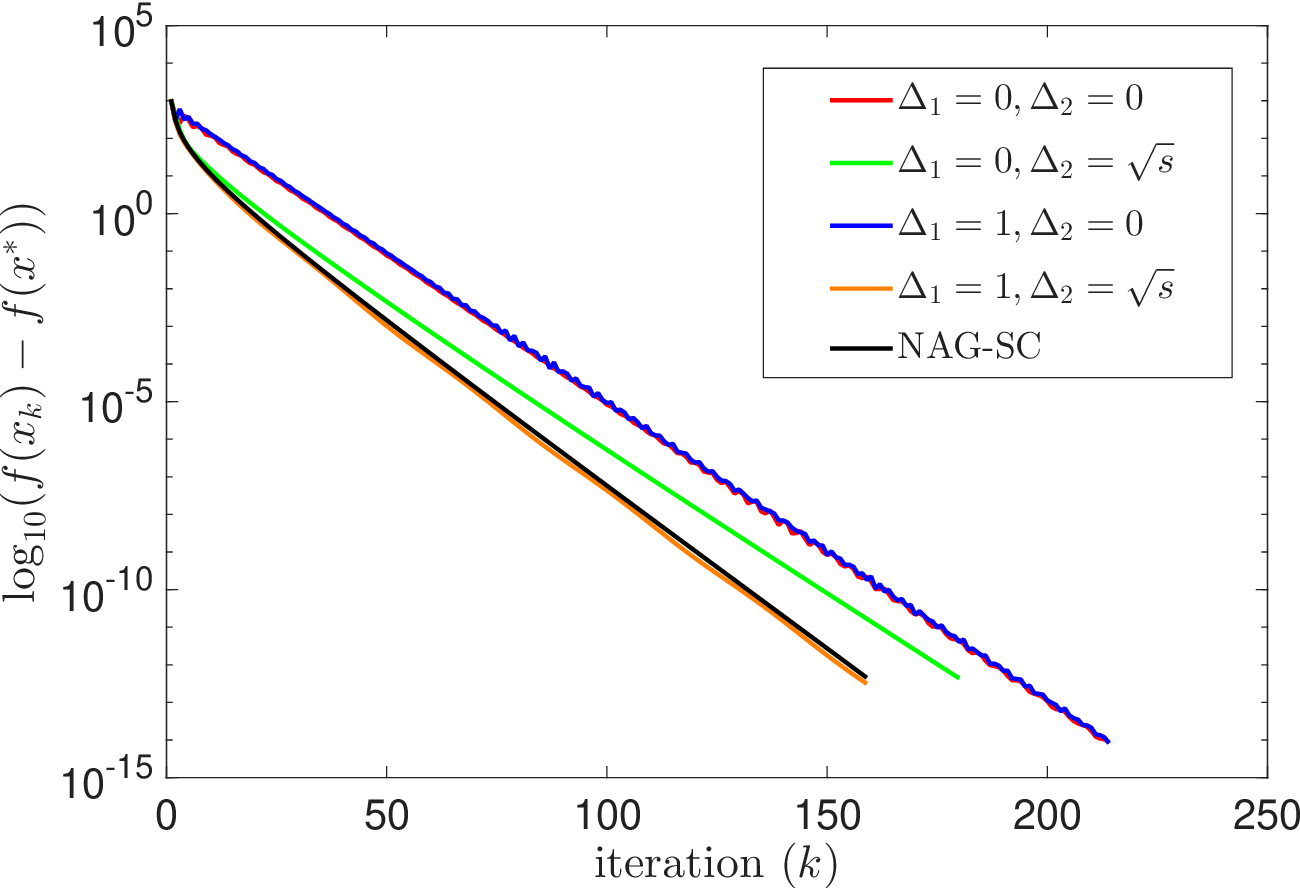}}
    \subfigure[$\hat{\Delta}_1=\sqrt{\mu s}$, $\hat{\Delta}_2=2\sqrt{s}/3$]{
    \centering
    \includegraphics[width=0.49\linewidth]{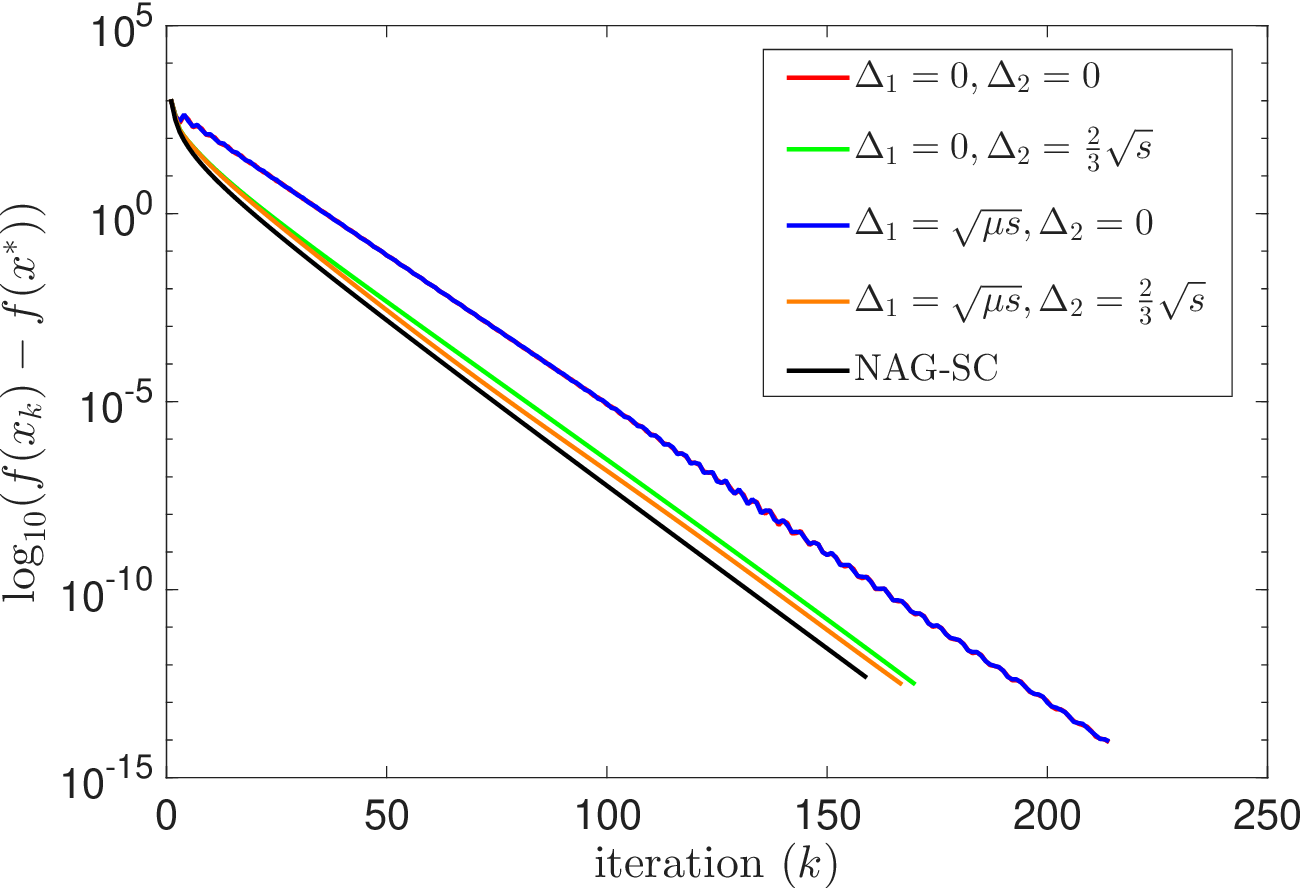}
    }    
    \caption{
    Numerical comparisons of scheme \eqref{eq: iteration of direct symplectic discretization} with different $(\widehat{\Delta}_1, \widehat{\Delta}_2)$ on solving problem \eqref{eq: quadratic function for numerical experiments} with $A$ given by \eqref{eq: case matrix}.}
    \label{pic:larger matrix}
\end{figure}

\end{document}